\documentclass[11pt,reqno,a4paper]{article} 

\usepackage{anysize}
\marginsize{3.0cm}{3.0cm}{2.2cm}{2.2cm}

\usepackage[english]{babel}
\usepackage[centertags]{amsmath}
\usepackage{amsfonts,amssymb,amsthm}
\usepackage{color}
\usepackage{hyperref}
\usepackage{mathrsfs}
\usepackage[sans]{dsfont}
\usepackage{tcolorbox}

\let\OLDthebibliography\thebibliography
\renewcommand\thebibliography[1]{
  \OLDthebibliography{#1}
  \setlength{\parskip}{0pt}
  \setlength{\itemsep}{0pt plus 0.3ex} }

\numberwithin{equation}{section}

\theoremstyle{plain}
\newtheorem{theorem}{Theorem}[section]

\newtheorem{lemma}[theorem]{Lemma}

\newtheorem{assumption}[theorem]{Assumption}

\theoremstyle{definition}
\newtheorem{remarkbase}[theorem]{Remark}
\newenvironment{remark}{\pushQED{\qed} \remarkbase}{\popQED\endremarkbase}

\renewcommand{\Im}{\mathrm{Im}\,}
\renewcommand{\Re}{\mathrm{Re}\,}

\newcommand{\N}{{\mathbb N}}
\newcommand{\R}{{\mathbb R}}
\newcommand{\C}{{\mathbb C}}

\newcommand{\mB}{\mathcal{B}}
\newcommand{\mC}{\mathcal{C}}

\newcommand{\mE}{\mathcal{E}}
\newcommand{\mF}{\mathcal{F}}
\newcommand{\mG}{\mathcal{G}}

\newcommand{\mL}{\mathcal{L}}

\newcommand{\mS}{\mathcal{S}}

\newcommand{\mN}{\mathcal{N}}

\newcommand{\mT}{\mathcal{T}}

\newcommand{\mZ}{\mathcal{Z}}

\renewcommand{\a}{\alpha}
\renewcommand{\b}{\beta}
\newcommand{\g}{\gamma}
\renewcommand{\d}{\delta}

\newcommand{\e}{\varepsilon}
\newcommand{\ph}{\varphi}
\newcommand{\lm}{\lambda}
\newcommand{\Lm}{\Lambda}

\newcommand{\om}{\omega}

\newcommand{\s}{\sigma}

\newcommand{\pa}{\partial}

\newcommand{\bcr}{\begin{color}{red}}
\newcommand{\ec}{\end{color}}
\newcommand{\bcb}{\begin{color}{blue}}
\newcommand{\bcg}{\begin{color}{green}}

\renewcommand{\oe}{\"o}

\newcommand{\op}{\mathrm{op}}

\newcommand{\aff}{\mathrm{a}} 

\newcommand{\bff}{\mathrm{b}} 
\newcommand{\bffb}{\bff}
\newcommand{\anm}{a} 
\newcommand{\sigmaa}{\sigma_\aff} 

\title{Size of data in implicit function problems 
and singular perturbations for nonlinear Schr\oe{}dinger systems} 

\author{\small{Pietro Baldi and Emanuele Haus}}
\date{} 
\begin{document}

\maketitle

\begin{small}
\textbf{Abstract.}
We investigate a general question about the size and regularity 
of the data and the solutions in implicit function problems with loss of regularity. 
First, we give a heuristic explanation of the fact that the optimal data size 
found by Ekeland and S\'er\'e with their recent non-quadratic version of the Nash-Moser theorem 
can also be recovered, for a large class of nonlinear problems, 
with quadratic schemes. 
Then we prove that this heuristic observation applies to 
the singular perturbation Cauchy problem 
for the nonlinear Schr\oe dinger system 
studied by M\'etivier, Rauch, Texier, Zumbrun, Ekeland, S\'er\'e. 
Using a ``free flow component'' decomposition 
and applying an abstract Nash-Moser-H\oe{}rmander theorem, 
we improve the existing results regarding both the size of the data 
and the regularity of the solutions.

\emph{MSC2010:} 47J07, 35B25, 35Q55, 35G55.
\end{small}

\bigskip

\noindent
\begin{minipage}{11cm}%
\begin{scriptsize}
\tableofcontents
\end{scriptsize}
\end{minipage}

\bigskip

\section{Introduction} 
\label{sec:intro}

This paper is motivated by a general question concerning the size and regularity 
of the data and the solutions in implicit function problems with loss of regularity. 
In the recent work \cite{ES}, Ekeland and S\'er\'e introduce a new iteration scheme 
in Banach spaces 
for solving nonlinear functional equations 
of the form 
\[
F(u) = v
\] 
where the linearized operator $F'(u)$ admits a right inverse that loses derivatives.  
In such situations, a well-established strategy for constructing a solution $u$ 
consists in applying a Nash-Moser iteration, 
essentially based on a quadratic Newton scheme combined with smoothing operators.  
The scheme in \cite{ES} differs from the standard Nash-Moser approach 
in that it is not quadratic, 
and it consists in solving a sequence of Galerkin problems 
by a topological argument (Ekeland's variational principle).
This gives two main improvements with respect to the standard quadratic approach: 
the map $F$ needs not be twice differentiable, 
and a larger ball for the datum $v$ is covered. 

The first point of the present paper is the observation that, 
for operators of the form 
\[
F(u) = L u + \mN(u)
\] 
where $L$ is linear and $\mN(u) = O(\|u\|^\a)$ for some $\a > 1$ 
in a ball $\| u \| \leq R$, 
the same size of the ball for the datum $v$ as in \cite{ES} 
can also be obtained by quadratic Nash-Moser schemes.
In Section \ref{sec:antani}
we explain the heuristics behind this simple, general observation.

In Sections \ref{sec:TZ}-\ref{sec:proof no loss}
we consider the singular perturbation Cauchy problem 
for the nonlinear Schr\oe dinger system 
studied by M\'etivier and Rauch \cite{MR}, 
Texier and Zumbrun \cite{TZ} 
and Ekeland and S\'er\'e \cite{ES}, 
and we rigorously prove that the observation 
of Section \ref{sec:antani} applies to this PDE problem. 
The result of Sections \ref{sec:TZ}-\ref{sec:proof no loss}
is stated in Theorem \ref{thm:no loss}, 
which improves the results in \cite{TZ} and \cite{ES} 
regarding 
the size of the data and also the regularity of the solution: 
for initial data in a Sobolev space $H^s(\R^d)$ 
we prove that the solution of the Cauchy problem 
belongs to $C([0,T], H^s(\R^d))$ with the \emph{same} regularity $s$,
as it is expected, and we give the corresponding estimate 
for the solution in terms of its initial datum. 
For initial data of a special ``concentrating'' form, see \eqref{1012.3},
Theorem \ref{thm:no loss}
also improves the size of the ball for the data
with respect to \cite{TZ} and \cite{ES}, 
see Remark \ref{rem:comparison}.

For initial data of the other special form considered in 
\cite{TZ} and \cite{ES} (``fast oscillating'' data, see \eqref{1012.3}), 
we improve the size of initial data in Theorem \ref{thm:with loss},
which is proved in Sections \ref{sec:L infty}-\ref{sec:proof with loss}.
With respect to Theorem \ref{thm:no loss}, the new ingredient 
is a ``free flow decomposition'' of the unknown, 
which is a natural way of exploiting the interplay 
between the linear and nonlinear part of the system
and the better $L^\infty$ embedding properties of concentrating or highly oscillating free flows 
(see Lemma \ref{lemma:free flow}),
inspired by the ``shifted map'' trick of \cite{TZ}. 
The price to pay for this improvement on the size of data 
is a loss of \emph{one} derivative: for data in $H^s(\R^d)$, 
the solution belongs to $C([0,T],H^{s-1}(\R^d))$.  
Theorem \ref{thm:with loss} improves the results of \cite{TZ} and \cite{ES}
both regarding the regularity of the solution
and the size of the data, see Remark \ref{rem:comparison}.

We point out that the loss of regularity in Theorem \ref{thm:with loss} 
is \emph{not} due to the Nash-Moser iteration: 
the loss of one derivative is introduced 
when solving the linearized Cauchy problem as a triangular system
(see \eqref{lin.with}) in two components, 
which are the ``free flow'' component of the unknown and its correction 
--- the Nash-Moser-H\oe{}rmander Theorem \ref{thm:NMH} 
just replicates the loss of one derivative for the nonlinear problem, 
without introducing additional losses. 
The loss of regularity in Theorem \ref{thm:with loss}
equals exactly the amount of derivatives in the nonlinearity, 
which is 1 in system \eqref{1012.1}.

The main difference between our ``free flow decomposition'' 
and the ``shifted map'' trick of \cite{TZ} is that we treat 
the free flow as an unknown, although it is already 
completely determined by the initial datum of the problem. 
In this way, Theorem \ref{thm:NMH} regularizes the free flow,
introducing just one new dyadic Fourier packet 
at each step of the iteration. 
This is the key ingredient for preserving the regularity 
of the linearized problem in the nonlinear one,
and it is somewhat reminiscent of a similar idea 
in H\oe{}rmander \cite{Geodesy}. 

Technical details of the fact that 
the heuristic observation of Section \ref{sec:antani} 
rigorously applies to Theorems \ref{thm:no loss} and \ref{thm:with loss} 
are contained in Remarks \ref{rem:confirm.1} and \ref{rem:confirm.2}.
Other general observations about the optimization of the data size in Nash-Moser schemes 
are in Remarks \ref{rem:best rescaling} and \ref{rem:specific}.

\medskip

\emph{Acknowledgements}. 
We warmly thank Ivar Ekeland and Eric S\'er\'e for many interesting discussions, 
in Naples and Paris, which have motivated this work. 

\begin{small}
Supported by INdAM -- GNAMPA Project 2019 
``Hamiltonian dynamics and evolution PDEs''
and PRIN 2015 
``Variational methods, with applications to problems in Mathematical Physics and Geometry''.
\end{small}

\section{Large radius with quadratic schemes: an informal explanation}
\label{sec:antani}

Consider a nonlinear problem of the kind 
\[
F(u) = v,
\]
where $v$ is given, $u$ is the unknown, and $F$ is a twice differentiable 
nonlinear operator in some Banach spaces satisfying $F(0) = 0$. Assume that 
for all $u$ in a ball $\| u \| \leq R$ 
the linearized operator $F'(u)$ 
admits a right inverse $\Psi(u)$ satisfying 
\begin{equation} \label{A}
\| \Psi(u) h \| \leq A \| h \|
\quad 
\forall \| u \| \leq R,
\end{equation}
and the second derivative $F''(u)$ satisfies
\begin{equation} \label{B}
\| F''(u)[h,w] \| \leq B \| h \| \| w \| 
\quad 
\forall \| u \| \leq R 
\end{equation}
(in this discussion we ignore completely the questions about loss of derivatives,
and we only care about size). 
As explained in \cite{ES}, the quadratic Newton scheme gives a solution $u$ 
of the equation $F(u) = v$ for all $v$ of size 
\[
\| v \| \lesssim \min \Big\{ \frac{1}{A^2 B} , \frac{R}{A} \Big\},
\]
while, with topological arguments, one can prove the existence of a solution $u$ 
for all $v$ in the larger ball
\[
\| v \| \lesssim \frac{R}{A}.
\]
Our observation is that, for operators $F$ in some large class, 
the two radii are of the same order.

Indeed, assume that $F$ is given by the sum of a linear part $\mL$ 
and a nonlinear one $\mN$, 
\[
F(u) = \mL u + \mN(u).
\]
Assume that $\mN$ satisfies 
\begin{align}
& \| \mN(u) \| \lesssim \| u \|^{p+1}, \notag \\
& \| \mN'(u) h \| \lesssim \| u \|^p \| h \|, \label{stima mN'} \\
& \| \mN''(u)[h,w] \| \lesssim \| u \|^{p-1} \| h \| \, \| w \| \label{stima mN''}
\end{align}
for some $p \geq 1$, for all $u$ in the ball $\| u \| \leq 1$, 
so that 
\[
\| F''(u)[h,w] \| \lesssim \| u \|^{p-1} \, \| h \| \, \| w \|.
\]
Suppose that $\mL$ has a right inverse $\mL_{r}^{-1}$ 
(namely $\mL \mL_r^{-1} = I$)
and that
\begin{equation} \label{kontrazia}
\| \mL_r^{-1} \mN'(u) \| \leq \frac12
\end{equation}
for $u$ sufficiently small, say $\| u \| \leq R$, 
so that, by Neumann series, the linearized operator 
\[
F'(u) = \mL + \mN'(u) = \mL (I + \mL_r^{-1} \mN'(u))
\]
has the right inverse 
\[
\Psi(u) = (I + \mL_r^{-1} \mN'(u))^{-1} \mL_r^{-1},
\]
with
\[
\| \Psi(u) \| \leq 2 \| \mL_r^{-1} \|.
\]
Hence \eqref{A} holds with 
\[
A := 2 \| \mL_r^{-1} \|.
\]
What is the ``intrinsic'' size of $R$? 
By \eqref{stima mN'}, condition \eqref{kontrazia} holds for 
\[
\| \mL_r^{-1} \| \| u \|^p \leq \frac12, \quad \text{i.e.} \quad 
\| u \| \leq \Big( \frac{1}{2 \| \mL_r^{-1} \|} \Big)^{\frac{1}{p}},
\]
therefore we fix 
\begin{equation} \label{R fixed}
R := \Big( \frac{1}{2 \| \mL_r^{-1} \|} \Big)^{\frac{1}{p}} = A^{- \frac{1}{p}}.
\end{equation}
Moreover, by \eqref{stima mN''}, condition \eqref{B} holds with 
\[
B := R^{p-1} = A^{-1 + \frac{1}{p}}.
\]
Thus 
\[
\frac{1}{A^2 B} = A^{-1 - \frac{1}{p}}, 
\quad 
\frac{R}{A} = A^{-1 - \frac{1}{p}},
\]
namely the two balls have the same size. 

\begin{remark}
Even when $\mL_r^{-1} \mN'(u)$ is an unbounded operator, 
so that the right invertibility of $F'(u)$ cannot be directly obtained by Neumann series, 
the heuristic argument above still catches the right size of $R$, 
provided that the invertibility of $F'(u)$ is obtained by a perturbative procedure.
\end{remark}

\section{Application to a singular perturbation problem}
\label{sec:TZ}

Like Ekeland and S\'er\'e in \cite{ES}, 
we consider the Cauchy problem studied by M\'etivier and Rauch \cite{MR} 
and Texier and Zumbrun \cite{TZ}, 
which is a nonlinear system of Schr\oe{}dinger equations arising in nonlinear optics. 
In \cite{MR}, M\'etivier and Rauch prove the existence of local solutions 
of the Cauchy problem, with existence time $T$ converging to $0$ when the 
Sobolev $H^s(\R^d)$ norm of the initial datum goes to infinity. 
In \cite{TZ}, Texier and Zumbrun use a Nash-Moser scheme 
to improve this result, giving a uniform lower bound for $T$ 
for two classes of initial data (concentrating and highly oscillating) 
whose $H^s(\R^d)$ norm goes to infinity. 
In \cite{ES}, Ekeland and S\'er\'e apply their non-quadratic version 
of the Nash-Moser theorem, 
extending the result in \cite{TZ} to even larger initial data.

Like in the aforementioned papers, 
we consider the system 
\begin{equation} \label{1012.1}
\pa_t v_j + i \lm_j \Delta v_j 
= \sum_{k = 1}^N 
\big( b_{jk}(v, \pa_x) v_{k} + c_{jk}(v,\pa_x) \overline{v_{k}} \big),
\quad \ j=1,\ldots,N, 
\end{equation}
where $v = v(t,x) = (v_1, \ldots, v_N) \in \C^N$ is 
the unknown, $(t,x) \in [0,T] \times \R^d$, 
$\lm_1, \ldots, \lm_N$ are constants, 
and $b_{jk}(v,\pa_x)$, $c_{jk}(v,\pa_x)$ are first order differential operators 
\begin{equation} \label{diff.str}
b_{jk}(v,\pa_x) = \sum_{\ell=1}^d b_{\ell jk}(v) \pa_{x_\ell}, 
\quad \ 
c_{jk}(v,\pa_x) = \sum_{\ell=1}^d c_{\ell jk}(v) \pa_{x_\ell}, 
\end{equation}
with $b_{\ell jk}, c_{\ell jk}$ complex-valued $C^\infty$ functions 
of $\Re(v_1), \ldots, \Re(v_N), \Im(v_1), \ldots, \Im(v_N)$
of order 
\begin{equation} \label{order p}
b_{\ell jk}(v) = O(|v|^p), \quad 
c_{\ell jk}(v) = O(|v|^p)
\end{equation}
in a ball around the origin, for some integer $p \geq 1$.

Following \cite{MR}, \cite{TZ} and \cite{ES}, 
we assume these ``transparency conditions'':

\begin{assumption} \label{transpa}
We assume that 

$(i)$ $\lm_1, \ldots, \lm_N$ are real and pairwise distinct;

$(ii)$ for all $j, k$ such that $\lm_j + \lm_{k} = 0$ 
there holds $c_{jk} = c_{kj}$; 

$(iii)$ for all $j$, $b_{jj}$ is real. 
\end{assumption}

Under these assumptions, the Cauchy problem for \eqref{1012.1}
is locally wellposed in the Sobolev space $H^s(\R^d)$ 
for $s > 1 + d/2$ (Theorem 1.5 in \cite{MR}). 
As is natural in the case of general initial data, 
the result in \cite{MR} gives an existence time $T$ going to 0 
as the initial datum goes to $\infty$ in $H^s(\R^d)$. 
In \cite{TZ} and \cite{ES} it is assumed that $p \geq 2$, 
and special initial data 
\begin{equation} \label{1012.2}
v(0,x) = \e^{\s} \aff_\e(x)
\end{equation} 
are considered, either concentrating or fast oscillating 
\begin{equation} \label{1012.3}
\aff_\e(x) = \aff_0(x/\e) \quad \text{(concentrating)}; 
\qquad \quad 
\aff_\e(x) = \aff_0(x) e^{i x \cdot \xi_0 / \e} \quad \text{(oscillating)},
\end{equation} 
with $\xi_0 \in \R^d$, 
and in both cases 
$0 < \e \leq 1$, $\s > 0$, 
$\aff_0 \in H^{s_1}(\R^d)$ for some large $s_1$. 

In \cite{TZ} and \cite{ES} the following results are proved. 

\begin{theorem}[Theorem 4.6 in \cite{TZ}] 
\label{thm:TZ}
Under the assumptions above, let $d,p \geq 2$ and 
\begin{equation} \label{TZ-ineq}
\s > \frac{k_c - \sigmaa - 1}{p+1}
\end{equation}
where $\sigmaa = d/2$ in the concentrating case, 
$\sigmaa = 0$ in the oscillating case, 
and $k_c$ is a constant depending on $(d,p)$.
For $s_1$ large enough, $T > 0$.  
If $\aff_0 \in H^{\bar s}(\R^d)$ for $\bar s$ large enough, 
and $\| \aff_0 \|_{H^{\bar s}}$ is small enough, 
then, for all $\e \in (0,1]$, 
the Cauchy problem \eqref{1012.1}-\eqref{1012.2}-\eqref{1012.3} 
has a unique solution in the space $C^1([0,T], H^{s_1 - 2}(\R^d)) 
\cap C^0([0,T], H^{s_1}(\R^d))$. 
\end{theorem}

The constant $k_c$ in \eqref{TZ-ineq} satisfies 
$k_c \geq \max \{ 6, k_c \geq 3 + \frac{dp}{2(p-1)} \}$, 
see Remark \ref{rem:comparison}.

\begin{theorem}[Theorem 6 in \cite{ES}] 
\label{thm:ES}
Under the assumptions above, let $d,p \geq 2$ and 
\begin{equation} \label{ES-ineq}
\s > \frac{d}{2} \frac{p}{p-1} - \sigmaa
\end{equation}
where $\sigmaa = d/2$ in the concentrating case, 
and $\sigmaa = 0$ in the oscillating case. 
Let $s_1 > d/2 + 4$ and $T > 0$.  
If $\aff_0 \in H^{\bar s}(\R^d)$ for $\bar s$ large enough, 
and $\| \aff_0 \|_{H^{\bar s}}$ is small enough, 
then, for all $\e \in (0,1]$, 
the Cauchy problem \eqref{1012.1}-\eqref{1012.2}-\eqref{1012.3} 
has a unique solution in the space $C^1([0,T], H^{s_1 - 2}(\R^d)) 
\cap C^0([0,T], H^{s_1}(\R^d))$. 
\end{theorem}

Following \cite{TZ}, we introduce the ``semi-classical'' 
Sobolev norms  
\begin{equation} \label{1012.4}
\| f \|_{H^s_\e} 
:= \| (-\e^2 \Delta + 1)^{s/2} f \|_{L^2(\R^d)}
= \| (1 + |\e \xi|^2)^{s/2} (\mF f)(\xi) \|_{L^2(\R^d_\xi)}, 
\quad \ s \in \R, 
\end{equation}
where $\mF$ is the Fourier transform on $\R^d$, and $0 < \e \leq 1$. 
The first theorem we prove in this paper is the following. 

\begin{theorem} \label{thm:no loss}
(i) \emph{(Existence)}
In the assumptions above, let $T > 0$, $p \geq 1$, $d \geq 1$, and $s_1 > d/2 + 4$. 
Then there exist constants $C, C' > 0$, $\e_0 \in (0,1]$,
depending on $T,p,d,s_1$ and on $\lm_j, b_{jk}, c_{jk}$ in system \eqref{1012.1},
such that 
for all $\e \in (0,\e_0]$, 
for all initial data $v_0 \in H^{s_1}(\R^d)$ in the ball 
\begin{equation} \label{NS.ball.v0}
\| v_0 \|_{H^{s_1}_\e} \leq C \e^q, \quad \ 
q := \frac{1}{p} + \frac{d}{2}, 
\end{equation}
the Cauchy problem for system \eqref{1012.1} with initial data $v(0,x) = v_0(x)$ 
has a solution 
\[
v \in C^0([0,T], H^{s_1}(\R^d)) \cap C^1([0,T], H^{s_1 - 2}(\R^d)),
\] 
which satisfies
\[
\sup_{t \in [0,T]} \| v(t) \|_{H^{s_1}_\e} 
+ \e^2 \sup_{t \in [0,T]} \| \pa_t v(t) \|_{H^{s_1-2}_\e} 
\leq C' \| v_0 \|_{H^{s_1}_\e}.
\]

(ii) \emph{(Higher regularity)}
If, in addition, $v_0 \in H^s(\R^d)$ for $s > s_1$, 
then 
\[
\sup_{t \in [0,T]} \| v(t) \|_{H^s_\e} 
+ \e^2 \sup_{t \in [0,T]} \| \pa_t v(t) \|_{H^{s-2}_\e} 
\leq C_s \| v_0 \|_{H^s_\e}
\]
where $C_s$ depends on $s$ (and it is independent of $\e, v_0, v$).

(iii) \emph{(Initial data of special form)}
In particular, initial data $v_0$ of the form 
\eqref{1012.2}-\eqref{1012.3}, with $\| \aff_0 \|_{H^{s_1}(\R^d)} \leq 1$, 
belong to the ball \eqref{NS.ball.v0} for all $\e$ sufficiently small 
if $\s + \sigmaa > q$, namely
\begin{equation} \label{soglia.no.loss}
\s > \frac{1}{p} + \frac{d}{2} - \sigmaa,
\end{equation}
where $\sigmaa = d/2$ in the concentrating case 
and $\sigmaa = 0$ in the oscillating case.
\end{theorem}

In the next theorem we deal with the case $p \geq 2$, 
where the power $p$ of the nonlinearity 
is used to improve the lower bound for $\s$, 
at the price of a loss of 1 derivative in the solution
with respect to the regularity of the datum. 

\begin{theorem}
\label{thm:with loss}
(i) \emph{(Existence)}
In the assumptions above, let $T > 0$, $p \geq 2$, $d \geq 1$, 
$s_1 > \max \{ d + 4, 6 \}$, 
and 
\begin{equation} \label{soglia.with.loss}
\s > \frac{1 + d/2 - \sigmaa}{p}
\end{equation}
where $\sigmaa = d/2$ in the concentrating case 
and $\sigmaa = 0$ in the oscillating case.

Then there exist constants $C > 0$, $\e_0 \in (0,1]$,
depending on $T,p,d,s_1$, on $\lm_j, b_{jk}, c_{jk}$ in system \eqref{1012.1},
and on the difference $\s - (1+d/2-\sigmaa)/p$, 
such that 
for all $\e \in (0,\e_0]$, 
for all functions $\aff_0 \in H^{s_1}(\R^d)$ in the ball 
\begin{equation} \label{NS.ball.a0}
\| \aff_0 \|_{H^{s_1}} \leq 1,
\end{equation}
the Cauchy problem for system \eqref{1012.1} 
with initial data of the form 
\eqref{1012.2}-\eqref{1012.3}
has a solution 
\[
v \in C^0([0,T], H^{s_1-1}(\R^d)) \cap C^1([0,T], H^{s_1 - 3}(\R^d))
\] 
on the time interval $[0,T]$. 
Such a solution $v$ is the sum 
\[
v = y + \tilde v
\] 
of a ``free flow'' component $y(t,x)$, 
which is the solution of the Cauchy problem for the free Schr\oe{}dinger system 
\[
\begin{cases}
\pa_t y_j + i \lm_j \Delta y_j = 0, \quad \ 
j = 1, \ldots, N, 
\\
y(0,x) = \e^\s \aff_\e(x),
\end{cases}
\]
and a ``correction'' term $\tilde v(t,x)$ satisfying 
$\tilde v(0,x) = 0$ and 
\[
\sup_{t \in [0,T]} \| \tilde v(t) \|_{H^{s_1-1}_\e} 
+ \e^2 \sup_{t \in [0,T]} \| \pa_t \tilde v(t) \|_{H^{s_1-3}_\e} 
\leq C \e^{\s + d/2} \| \aff_0 \|_{H^{s_1}}.
\]

(ii) \emph{(Higher regularity)}
If, in addition, $\aff_0 \in H^s(\R^d)$ for $s > s_1$, 
then 
\[
\sup_{t \in [0,T]} \| \tilde v(t) \|_{H^{s-1}_\e} 
+ \e^2 \sup_{t \in [0,T]} \| \pa_t \tilde v(t) \|_{H^{s-3}_\e} 
\leq C_s \e^{\s + d/2} \| \aff_0 \|_{H^s}
\]
where $C_s$ depends on $s$ (and it is independent of $\e, \aff_0$).
\end{theorem}

\begin{remark}[\emph{Smallness in low norm}]
 \label{rem:palla norma bassa}
In the higher regularity case, the smallness assumptions \eqref{NS.ball.v0}
in Theorem \ref{thm:no loss} and \eqref{NS.ball.a0} in Theorem \ref{thm:with loss}
are only required in the low norm $s_1$, with radii independent of the high regularity $s$.
\end{remark}

\begin{remark}[\emph{Comparison with previous results}]
\label{rem:comparison}
As observed in \cite{TZ} and \cite{ES}, 
M\'etivier and Rauch \cite{MR} already provide existence for a fixed positive $T$, uniformly in $\e$, 
when
\[
\s \geq \s_{\text{MR}} := 1 + d/2 - \sigmaa.
\]
Hence \cite{TZ}, \cite{ES} and Theorems \ref{thm:no loss}-\ref{thm:with loss} 
give something new only for $\s < \s_{\text{MR}}$.

The result of Texier and Zumbrun holds for $d \geq 2$, $p \geq 2$, 
and $\s$ above the threshold 
\[
\s_{\text{TZ}} := \frac{k_c - \sigmaa - 1}{p+1}
\]
(Theorem 4.6 in \cite{TZ}), 
where the constant $k_c$ satisfies some conditions;
in particular, $k_c \geq 6$ and 
\[
k_c \geq 3 + \frac{d}{2} \frac{p}{p-1},
\]
whence
\[
\s_{\text{TZ}} \geq \frac{1}{p+1} \Big( 2 + \frac{d}{2} \frac{p}{p-1} - \sigmaa \Big)
=: c.
\]
The threshold for $\s$ in our Theorem \ref{thm:with loss} is
\[
\s^*_1 := \frac{1 + d/2 - \sigmaa}{p} = \frac{\s_{\text{MR}}}{p}.
\]
For all pairs $(d,p)$ covered by \cite{TZ} 
(namely $d,p \geq 2$), one has $\s^*_1 < c \leq \s_{\text{TZ}}$,  
therefore we get a larger ball for the initial data.    
More precisely, regarding the data size, 
the improvement of Theorem \ref{thm:with loss} with respect to \cite{TZ} 
corresponds to the exponent $\s$ in the interval 
$\s^*_1 < \s \leq \min \{ \s_{\text{TZ}}, \s_{\text{MR}} \}$.
Note that for some pairs $(d,p)$ one has $\s_{\text{TZ}} \geq \s_{\text{MR}}$ 
(see Examples 4.8-4.9 in \cite{TZ}), 
so that \cite{TZ} gives no improvements with respect to \cite{MR};
our result improves \cite{MR} also in those cases. 

The result of Ekeland and S\'er\'e holds for $d,p \geq 2$,
and $\s$ above the threshold 
\[
\s_{\text{ES}} := \frac{d}{2} \frac{p}{p-1} - \sigmaa
\]
(Theorem 6 in \cite{ES}).  
The threshold for $\s$ in our Theorem \ref{thm:no loss} is
\[
\s^*_0 := \frac{1}{p} + \frac{d}{2} - \sigmaa.
\]
Since $\s^*_0 < \s_{\text{ES}}$ for all $d,p \geq 2$,
we get a larger ball for the initial data 
also with respect to \cite{ES}.

With respect to \cite{TZ} and \cite{ES} 
we also improve the regularity of the solution 
with respect to that of the initial data:
using Theorem \ref{thm:with loss}, 
the solution is one derivative less regular than the data
(the loss of regularity is one), 
while with Theorem \ref{thm:no loss} 
the solution has the same regularity as the data
(the loss is zero).
In \cite{TZ} and \cite{ES}, instead, 
the loss of regularity 
depends in a nontrivial way on several parameters of the iteration scheme, 
it blows up to $+ \infty$ in certain parameter regimes, 
and, in particular, can never be zero.
\end{remark}

\section{Functional setting} 
\label{sec:FS}

In this section we introduce weighted Sobolev norms 
and recall the basic inequalities 
that will be used in the rest of the paper. 

For $s \in \R$, we define 
\begin{equation} \label{FS.0}
\| u \|_{H^s(\R^d)} := \| \Lm^s u \|_{L^2(\R^d)}, 
\quad \ 
\| u \|_{H^s_\e(\R^d)} := \| \Lm^s_\e u \|_{L^2(\R^d)}, 
\end{equation}
where $\Lm^s = (1 - \Delta)^{s/2}$ is the Fourier multiplier 
of symbol $(1 + |\xi|^2)^{s/2}$ 
and $\Lm^s_\e = (1 - \e^2 \Delta)^{s/2}$ 
is that of symbol $(1 + \e^2 |\xi|^2)^{s/2}$,  
namely, following \cite{TZ}, 
\begin{equation} \label{FS.1} 
\| u \|_{H^s_\e(\R^d)} 
= \| (1-\e^2 \Delta)^{s/2} u \|_{L^2(\R^d)}
= \| (1 + |\e \xi|^2)^{s/2} \hat u(\xi) \|_{L^2(\R^d_\xi)}, 
\quad \ s \in \R, 
\end{equation}
where $\hat u$ is the Fourier transform of $u$ on $\R^d$, 
and $0 < \e \leq 1$. 
For all $u \in H^s(\R^d)$, one has 
\begin{equation} \label{FS.2}
\widehat{(R_\e u)}(\xi) 
= \e^{-d} \, \widehat u(\e^{-1} \xi), 
\quad \ 
(R_\e u)(x) := u(\e x),
\end{equation}
whence 
\begin{equation} \label{FS.3}
\Lm^s R_\e = R_\e \Lm^s_\e, 
\quad \ 
\| u \|_{H^s_\e(\R^d)} = \e^{d/2} \| R_\e u \|_{H^s(\R^d)}.
\end{equation}
We define the scalar product 
\begin{equation} \label{FS.4}
\langle u, v \rangle_{H^s_\e(\R^d)} 
:= \langle \Lm^s_\e u , \Lm^s_\e v \rangle_{L^2(\R^d)}.
\end{equation}
To shorten the notation, we write $\| \, \|_{H^s}$ instead of $\| \, \|_{H^s(\R^d)}$, 
and so on. 
Using \eqref{FS.3}, it is immediate to obtain the Sobolev embedding 
and the standard tame estimates for products and compositions of functions 
in terms of the rescaled norms \eqref{FS.1}: 
for the Sobolev embedding, one has 
\begin{equation} \label{FS.emb}
\| u \|_{L^\infty} 
= \| R_\e u \|_{L^\infty}
\leq C_{s_0} \| R_\e u \|_{H^{s_0}} 
= C_{s_0} \e^{-d/2} \| u \|_{H^{s_0}_\e}
\end{equation}
for all $s_0 > d/2$, all $u \in H^{s_0}(\R^d)$, 
for some constant $C_{s_0}$ depending on $s_0, d$;  
for the product, one has 
\begin{equation} \label{FS.prod}
\| u v \|_{H^s_\e} 
\leq C_s ( \| u \|_{L^\infty} \| v \|_{H^s_\e} + \| u \|_{H^s_\e} \| v \|_{L^\infty})
\end{equation}
for all $u, h \in H^s(\R^d)$, all $s \geq 0$, 
for some constant $C_s$ depending only on $s,d$;
for the composition, given any $C^\infty$ function $f$ 
such that $f(y) = O(y^p)$ around the origin 
for some integer $p \geq 1$, one has
\begin{equation} \label{FS.comp}
\| f(u) \|_{H^s_\e} 
\leq C_{s,M} \| u \|_{L^\infty}^{p-1} \| u \|_{H^s_\e}
\end{equation}
for all $M > 0$, 
all $u \in H^s(\R^d)$ in the ball $\| u \|_{L^\infty} \leq M$, 
all $s \geq 0$, 
for some constant $C_{s,M}$ depending only on $s,M,d,f$. 
Moreover, 
\begin{equation} \label{FS.der}
\e^{|\a|} \| \pa_x^\a u \|_{H^s_\e} \leq \| u \|_{H^{s+|\a|}_\e}
\end{equation}
for all multi-indices $\a \in \N^d$. 

For $m \geq 0$ integer, we define 
\begin{equation} \label{FS.5}
\| u \|_{W^{m,\infty}} 
:= \sum_{\begin{subarray}{c} \a \in \N^d \\ |\a| \leq m \end{subarray}} 
\| \pa_x^\a u \|_{L^\infty},
\quad \ 
\| u \|_{W^{m,\infty}_\e} 
:= \sum_{\begin{subarray}{c} \a \in \N^d \\ |\a| \leq m \end{subarray}} 
\e^{|\a|} \| \pa_x^\a u \|_{L^\infty}.
\end{equation}
One has
\begin{equation} \label{FS.6}
\pa_x^\a R_\e = \e^{|\a|} R_\e \pa_x^\a,
\quad\ 
\| u \|_{W^{m,\infty}_\e} = \| R_\e u \|_{W^{m,\infty}}. 
\end{equation}
Similarly as \eqref{FS.comp}, 
given any $C^\infty$ function $f$ such that $f(y) = O(y^p)$ 
around the origin for some positive integer $p$, one has
\begin{equation} \label{FS.7}
\| f(u) \|_{W^{m,\infty}_\e} 
\leq C_{m,M} \| u \|_{L^\infty}^{p-1} \| u \|_{W^{m,\infty}_\e}
\end{equation}
for all $M > 0$, 
all $u \in W^{m,\infty}(\R^d)$ in the ball $\| u \|_{L^\infty} \leq M$, 
all integers $m \geq 0$, 
for some constant $C_{m,M}$ depending on $m,M,d,f$. 
For the product of two functions, we also have 
\begin{equation} \label{FS.ss0}
\| uv \|_{H^s_\e} 
\leq \e^{-d/2} (C_{s_0} \| u \|_{H^{s_0}_\e} \| v \|_{H^s_\e} 
+ C_s \| u \|_{H^s_\e} \| v \|_{H^{s_0}_\e})
\end{equation} 
for all $s \geq 0$, $s_0 > d/2$, 
all $u,v \in H^s(\R^d) \cap H^{s_0}(\R^d)$,
and
\begin{equation} \label{FS.bona}
\| uv \|_{H^s_\e} 
\leq 2 \| u \|_{L^\infty} \| v \|_{H^s_\e} 
+ C_s \| u \|_{W^{m,\infty}_\e} \| v \|_{L^2}
\end{equation} 
for all $s \geq 0$, 
all $v \in H^s(\R^d)$, 
all $u \in W^{m,\infty}(\R^d)$, 
where $m$ is the smallest positive integer 
such that $m \geq s$, and $C_s$ depends on $s,d$. 
Estimate \eqref{FS.bona} is proved in the Appendix 
(see \eqref{bona} in Lemma \ref{lemma:prod s real}).
We remark that the constants $C_{s_0}, C_s, C_{s,M}, C_{m,M}$ 
in \eqref{FS.emb}, \eqref{FS.prod}, \eqref{FS.comp}, \eqref{FS.7}, 
\eqref{FS.ss0}, \eqref{FS.bona}
are independent of $\e$, 
and $C_{s_0}$ is also independent of $s$.

For time-dependent functions $u(t,x)$, $t \in [0,T]$,  
we denote, in short, 
\begin{alignat}{2} 
\label{def C Hs}
\| u \|_{C^0 H^s_\e} 
& := \| u \|_{C([0,T], H^s_\e)}, 
& \qquad 
\| u \|_{C^1_\e H^s_\e} 
& := \| u \|_{C^0 H^s_\e} + \e^2 \| \pa_t u \|_{C^0 H^{s-2}_\e},
\\
\| u \|_{C^0 W^m_\e} 
& := \| u \|_{C([0,T], W^{m, \infty}_\e)}, 
& \qquad  
\| u \|_{C^1_\e W^m_\e} 
& := \| u \|_{C^0 W^m_\e} + \e^2 \| \pa_t u \|_{C^0 W^{m-2}_\e}.
\label{def C Wm}
\end{alignat}
The notation $a \lesssim_s b$ 
means $a \leq C_s b$ for some constant $C_s$,  
independent of $\e$, possibly depending on $s$; 
also, $a \lesssim b$ means $a \leq C b$ for some constant $C$ 
independent of $\e$ and $s$.

\section{Analysis of the singular perturbation problem}
\label{sec:general}

In \cite{TZ} and \cite{ES}, system \eqref{1012.1} is written as 
\begin{equation} \label{0701.1}
\pa_t u + i A(\pa_x) u = B(u, \pa_x) u
\end{equation}
where $u = (v, \overline{v}) = (v_1, \ldots, v_N, \overline{v_1}, \ldots, \overline{v_N})$ 
is the unknown, 
$A(\pa_x)$ is the constant coefficients operator of second order
\[
A(\pa_x) = \mathrm{diag}(\lm_1, \ldots, \lm_n, - \lm_1, \ldots, - \lm_n) \Delta,
\]
$B(u, \pa_x)$ is the operator matrix 
\[
B = \begin{pmatrix} \mB & \mC \\ 
\overline{\mC} & \overline{\mB} \end{pmatrix},
\]
$\mB, \mC$ are the operator matrices with entries $b_{jk}(v,\pa_x)$, $c_{jk}(v,\pa_x)$ respectively, 
and $\overline{\mB}, \overline{\mC}$ have conjugate entries coefficients. 
To deal with concentrating or highly oscillating initial data \eqref{1012.3},  
in \cite{TZ} the weighted Sobolev norms \eqref{FS.1} are introduced. 
Recalling \eqref{FS.der}, 
it is natural, 
as it is done in \cite{TZ} and \cite{ES}, 
to write the powers of $\e$ as separate factors, 
writing \eqref{0701.1} as 
\begin{equation} \label{0701.2}
\pa_t u + i \e^{-2} A(\e \pa_x) u = \e^{-1} B(u, \e \pa_x) u
\end{equation}
where $A(\e \pa_x) := \e^2 A(\pa_x)$ and 
$B(u, \e \pa_x) := \e B(u,\pa_x)$. 
In this way $A(\e \pa_x)$ and $B(u, \e \pa_x)$ 
satisfy estimates that are uniform in $\e$:
\begin{equation} \label{0701.A}
\| A(\e \pa_x) u \|_{H^s_\e} 
\leq C_0 \| u \|_{H^{s+2}_\e} 
\end{equation}
for all $s \in \R$, all $u \in H^s(\R^d)$, 
with $C_0 = \max \{ |\lm_1|, \ldots, |\lm_N| \}$; 
\begin{equation} \label{0701.B}
\| B(u, \e \pa_x) h \|_{H^s_\e} 
\leq C_s (\| u \|_{L^\infty}^p \| h \|_{H^{s+1}_\e} 
+ \| u \|_{L^\infty}^{p-1} \| u \|_{H^s_\e} \| \e \pa_x h \|_{L^\infty})
\end{equation}
for all $s \geq 0$, 
all $h \in H^{s+1}(\R^d)$,
all $u \in H^s(\R^d)$ 
in the ball $\| u \|_{L^\infty} \leq 1$;
also, by \eqref{FS.bona} and \eqref{FS.7}, 
\begin{align} 
\| B(u, \e \pa_x) h \|_{H^s_\e} 
& \leq C \| u \|_{L^\infty}^p \| h \|_{H^{s+1}_\e} 
+ C_s \| u \|_{L^\infty}^{p-1} \| u \|_{W^{[s]+1,\infty}_\e} \| h \|_{H^1_\e},
\label{B generica}
\end{align}
for all $s \geq 0$, 
all $h \in H^{s+1}(\R^d)$,
all $u \in W^{[s]+1,\infty}(\R^d)$ 
in the ball $\| u \|_{L^\infty} \leq 1$,
where $[s]$ is the integer part of $s$;
and, by \eqref{prod pax} and \eqref{FS.7}, 
\begin{equation}\label{B -1}
\| B(u, \e \pa_x) h \|_{H^s_\e} 
\leq C \| u \|_{L^\infty}^{p-1} \| u \|_{W^{1,\infty}_\e} \| h \|_{H^{s+1}_\e} 
\end{equation}
for all $-1 \leq s \leq 0$, 
all $h \in H^{s+1}(\R^d)$,
all $u \in W^{1,\infty}(\R^d)$ 
in the ball $\| u \|_{L^\infty} \leq 1$.
The constants in \eqref{0701.A}, \eqref{0701.B}, \eqref{B generica}, \eqref{B -1} 
do not depend on $\e \in (0,1]$; 
$C_0, C$ in \eqref{0701.A}, \eqref{B generica} and \eqref{B -1}
are also independent of $s$.  

We consider the Cauchy problem for \eqref{0701.2} 
with initial data \eqref{1012.2}, 
namely 
\begin{equation} \label{Cp.general}
\begin{cases} 
\pa_t u + P(u) = 0, \\ 
u(0) = u_0
\end{cases}
\end{equation}
where
\begin{equation} \label{def P u0}
P(u) := i \e^{-2} A(\e \pa_x) u - \e^{-1} B(u, \e \pa_x) u,
\qquad 
u_0(x) := \e^\s (\aff_\e(x), \overline{\aff_\e(x)}).
\end{equation}
To apply our Nash-Moser theorem, 
we need to construct a right inverse for the linearized problem
and to estimate the second derivative of the nonlinear operator.
Let us begin with the linear inversion problem.
\medskip

\textbf{Analysis of the linearized problem.} 
Given $u(t,x)$, $f_1(t,x)$ and $f_2(x)$, 
consider the linear Cauchy problem for the unknown $h(t,x)$ 
\begin{equation} \label{Cp.lin}
\begin{cases} 
\pa_t h + P'(u)h = f_1, \\ 
h(0) = f_2,
\end{cases}
\end{equation}
where
\begin{align} \label{0901.2} 
P'(u) h & \,
= i \e^{-2} A(\e \pa_x) h - \e^{-1} B(u, \e \pa_x) h 
+ R_0(u) h,
\\
R_0(u) h & := - \e^{-1} (\pa_u B)(u, \e \pa_x)[h] u.
\label{def R0} 
\end{align}
Following \cite{TZ}, let 
\[
J := \{ (j,k) : \lm_j + \lm_k = 0 \}, 
\]
and let $\chi \in C^\infty_c(\R^d, \R)$ be a frequency truncation 
such that $0 \leq \chi(\xi) \leq 1$, 
$\chi(\xi) = 1$ for $|\xi| \leq 1/2$, 
and $\chi(\xi) = 0$ for $|\xi| \geq 1$. 
Like in \cite{TZ}, we decompose $B$ into the sum of a resonant term, 
a non-resonant term, and a low-frequency term: 
$B = B_r + B_{nr} +B_{lf}$, where
\begin{itemize}
\item 
the resonant term is 
\[
B_r := \begin{pmatrix} \mB_d & \mC_J \\ 
\overline{\mC_J} & \mB_d \end{pmatrix}
\]
where $\mB_d := \mathrm{diag}(b_{11}, \ldots, b_{NN})$,
$(\mC_J)_{jk} := c_{jk}$ if $(j,k) \in J$
and $(\mC_J)_{jk} := 0$ otherwise. 
By Assumption \ref{transpa}, the matrix $B_r(v,\xi)$ is Hermitian; 

\item
the nonresonant term is 
\[
B_{nr} := \begin{pmatrix} \mB^1 & \mC^1 \\ 
\overline{\mC}^1 & \overline{\mB}^1 \end{pmatrix}
\]
where $(\mB^1)_{jk} := (1 - \chi) b_{jk}$ if $j \neq k$, 
and $(\mB^1)_{jk} := 0$ if $j = k$; 
$(\mC^1)_{jk} := (1 - \chi) c_{jk}$ if $(j,k) \notin J$, 
and $(\mC^1)_{jk} := 0$ if $(j,k) \in J$; 

\item
the low-frequency term is 
\[
B_{lf} := \begin{pmatrix} \mB^0 & \mC^0 \\ 
\overline{\mC}^0 & \overline{\mB}^0 \end{pmatrix}
\]
where $(\mB^0)_{jk} := \chi b_{jk}$ if $j \neq k$, 
and $(\mB^0)_{jk} := 0$ if $j = k$; 
$(\mC^0)_{jk} := \chi c_{jk}$ if $(j,k) \notin J$, 
and $(\mC^0)_{jk} := 0$ if $(j,k) \in J$.  
\end{itemize}
We recall the normal form transformation of \cite{TZ}
(see the proof of Lemma 4.5 in \cite{TZ}): 
define the pseudo-differential matrix symbol $M(u(t,x),\xi)$ as 
\begin{equation} \label{def M}
M_{jk}(u(t,x), \xi) := \begin{cases} \dfrac{(B_{nr})_{jk}(u(t,x), i\xi)}{i |\xi|^2 (\om_j - \om_k)} 
& \text{if} \ \om_j \neq \om_k, 
\\
0 & \text{if} \ \om_j = \om_k,
\end{cases}
\end{equation}
where 
\[
\om_j := \begin{cases}
- \lm_j \quad & \text{for} \ j = 1, \ldots, N, 
\\
\lm_{j-N} \quad & \text{for} \ j = N+1, \ldots, 2N.
\end{cases}
\]
Since the commutator of $A$ and $M$ is the matrix
\begin{equation} \label{homo.eq}   
[ A(i\xi), M(u, \xi) ] 
= \big( |\xi|^2 (\om_j - \om_k) M_{jk}(u, \xi) \big)_{j,k = 1, \ldots, 2N},
\end{equation}
one has 
\[
B_{nr}(u(t,x), i \xi) - i [ A(i\xi), M(u(t,x), \xi) ] = 0.
\]
Like in \cite{TZ}, we introduce the following semiclassical quantization 
of a symbol $\s(x,\xi)$
\[
\op_\e(\s)h (x) := (2\pi)^{-d/2} 
\int_{\R^d} \s(x, \e \xi) \hat h(\xi) e^{i \xi \cdot x} \, d\xi.
\]
By \eqref{def M} and \eqref{B generica}, one has
\begin{align}  
\| \op_\e(M) h \|_{H^s_\e} 
& \leq C \| u \|_{L^\infty}^p \| h \|_{H^{s-1}_\e} 
+ C_s \| u \|_{L^\infty}^{p-1} \| u \|_{W^{[s]+1,\infty}_\e} \| h \|_{H^{-1}_\e},
\label{pasubio.11}
\\
\| \op_\e(M) h \|_{L^2} 
& \leq C \| u \|_{L^\infty}^p \| h \|_{H^{-1}_\e}
\label{pasubio.-1-1}
\end{align}
for all $s \geq 0$, all $\| u \|_{L^\infty} \leq 1$, all $h$. 
Hence there exists $\rho_0 > 0$, independent of $\e$, 
such that, for $u$ in the ball 
\begin{equation} \label{ball for M}
\e \| u \|_{L^\infty}^p \leq \rho_0,
\end{equation}
one has 
\begin{equation} \label{pasubio.1/2}
\| \e \op_\e(M) h \|_{H^{-1}_\e} 
\leq \| \e \op_\e(M) h \|_{L^2} 
\leq C \e \| u \|_{L^\infty}^p \| h \|_{H^{-1}_\e} 
\leq \tfrac12 \| h \|_{H^{-1}_\e}
\leq \tfrac12 \| h \|_{L^2}.
\end{equation}
Therefore, by Neumann series, $I + \e \op_\e(M)$ is invertible
in $H^{-1}_\e$ and in $L^2$, and 
\begin{equation} \label{pasubio.12}
\| (I + \e \op_\e(M))^{-1} h \|_{H^s_\e} 
\leq C \| h \|_{H^s_\e} 
+ C_s \e \| u \|_{L^\infty}^{p-1} \| u \|_{W^{[s]+1,\infty}_\e} \| h \|_{H^{-1}_\e}
\end{equation}
for all $s \geq 0$, for $u \in W^{[s]+1,\infty}_\e(\R^d)$ 
in the ball \eqref{ball for M} (where $\rho_0$ is independent of $s$). 

Under the change of variable 
\begin{equation} \label{change h ph}
h = (I + \e \op_\e(M)) \ph,
\end{equation}
the linear Cauchy problem \eqref{Cp.lin} becomes 
\begin{equation} \label{Cp.Q}
\begin{cases} 
\pa_t \ph + Q(u) \ph = g_1 \\ 
\ph(0) = g_2
\end{cases}
\end{equation}
where 
\begin{equation} \label{def g1 g2}
g_1 := (I + \e \op_\e(M))^{-1} f_1, 
\quad 
g_2 := (I + \e \op_\e(M))^{-1}\big|_{t=0} f_2, 
\end{equation}
and, by \eqref{homo.eq},   
\begin{align}
\pa_t + Q(u) & :=  (I + \e \op_\e(M))^{-1} (\pa_t + P'(u)) (I + \e \op_\e(M)) 
\notag \\
& \, = \pa_t + i \e^{-2} A(\e \pa_x) - \e^{-1} B_r(u, \e \pa_x) + G(u),
\label{def Q}
\end{align}
with
\begin{align}
G(u) & := 
(I + \e \op_\e(M))^{-1} \Big( \e \op_\e(M) \e^{-1} B_r(u, \e \pa_x)
- \e^{-1} B_{lf}(u, \e \pa_x) 
\notag \\ & \quad \ 
+ \e \op_\e(\pa_t M) 
- \e^{-1} B(u, \e \pa_x) \e \op_\e(M) 
+ R_0(u) (I + \e \op_\e(M)) \Big)
\label{def G}
\end{align}
(we have used the trivial identity 
$I - (I + K)^{-1} = (I + K)^{-1} K$
for $K = \e \op_\e(M)$). 

Now we prove an energy estimate for \eqref{def Q}, 
and we start with the term $G(u)$. 
By \eqref{pasubio.12}, 
\eqref{pasubio.11}, 
\eqref{pasubio.1/2}, 
\eqref{B generica},
\eqref{B -1},
the first term in \eqref{def G} satisfies, for $s \geq 0$,
\begin{multline*}
\| (I + \e \op_\e(M))^{-1} \e \op_\e(M) \e^{-1} B_r(u, \e \pa_x) \ph \|_{H^s_\e} 
\\ 
\lesssim_s \| u \|_{L^\infty}^{2p} \| \ph \|_{H^s_\e} 
+ \| u \|_{L^\infty}^{2p-2} \| u \|_{W^{[s]+1,\infty}_\e} \| u \|_{W^{1,\infty}_\e} \| \ph \|_{L^2}.
\end{multline*} 
and  
\[
\| (I + \e \op_\e(M))^{-1} \e \op_\e(M) \e^{-1} B_r(u, \e \pa_x) \ph \|_{L^2} 
\lesssim \| u \|_{L^\infty}^{2p-1} \| u \|_{W^{1,\infty}_\e} \| \ph \|_{L^2}.
\]
The low-frequency term $B_{lf}$ satisfies, for $s \geq 0$,  
\[
\| \e^{-1} B_{lf}(u, \e \pa_x) \ph \|_{H^s_\e} 
\lesssim_s \e^{-1} \| u \|_{L^\infty}^{p-1} \| u \|_{W^{[s]+1,\infty}_\e} \| \ph \|_{L^2}.
\]
The term containing the time derivative of the symbol $M$ is estimated,
for $s \geq 0$, by
\begin{align*}
\| \e \op_\e(\pa_t M) \ph \|_{H^s_\e} 
& \lesssim_s 
\e \| u \|_{L^\infty}^{p-1} \| \pa_t u \|_{L^\infty} \| \ph \|_{H^{s-1}_\e} 
\\ & \quad \ 
+ \e ( \| u \|_{L^\infty}^{p-1} \| \pa_t u \|_{W^{[s]+1,\infty}_\e}
+ \| u \|_{L^\infty}^{\nu} \| u \|_{W^{[s]+1,\infty}_\e} \| \pa_t u \|_{L^\infty} ) 
\| \ph \|_{H^{-1}_\e}
\end{align*}
where
\begin{equation} \label{def nu}
\nu := \max \{ p-2 , 0 \},
\end{equation}
and, by \eqref{def M}, 
\[
\| \e \op_\e(\pa_t M) \ph \|_{H^{-1}_\e} 
\leq \| \e \op_\e(\pa_t M) \ph \|_{L^2} 
\lesssim \e \| u \|_{L^\infty}^{p-1} \| \pa_t u \|_{L^\infty} \| \ph \|_{H^{-1}_\e}.
\] 
Next, $R_0$ defined in \eqref{def R0} satisfies, for $s \geq 0$, 
\begin{align}
\| R_0(u) \ph \|_{H^s_\e} 
& \lesssim_s 
\e^{-1} \| u \|_{L^\infty}^{p-1} ( \| u \|_{W^{1,\infty}_\e} \| \ph \|_{H^s_\e}
+ \| u \|_{W^{[s]+2,\infty}_\e} \| \ph \|_{L^2}),
\label{est R0 Hs}
\\
\| R_0(u) \ph \|_{L^2} 
& \lesssim 
\e^{-1} \| u \|_{L^\infty}^{p-1} \| u \|_{W^{1,\infty}_\e} \| \ph \|_{L^2}.
\label{est R0 L2} 
\end{align}
Hence $G(u)$ in \eqref{def G} satisfies, for all $s \geq 0$, 
\begin{align}
\| G(u) \ph \|_{H^s_\e} 
& \lesssim_s 
\e^{-1} \| u \|_{L^\infty}^{p-1} 
( \| u \|_{W^{1,\infty}_\e} + \e^2 \| \pa_t u \|_{L^\infty} ) \| \ph \|_{H^s_\e}
\notag \\ & \quad \ \  
+ \e^{-1} \{ \| u \|_{L^\infty}^{p-1} 
(\| u \|_{W^{[s]+2,\infty}_\e} + \e^2 \| \pa_t u \|_{W^{[s]+1,\infty}_\e})  
\notag \\ & \qquad \qquad \quad
+ \e^2 \| u \|_{L^\infty}^\nu \| u \|_{W^{[s]+1,\infty}_\e} 
\| \pa_t u \|_{L^\infty} \} \| \ph \|_{L^2},
\label{est G}
\\
\label{est G L2}
\| G(u) \ph \|_{L^2} 
& \lesssim \e^{-1} \| u \|_{L^\infty}^{p-1} 
( \| u \|_{W^{1,\infty}_\e} + \e^2 \| \pa_t u \|_{L^\infty} ) \| \ph \|_{L^2}.
\end{align}

The constant coefficient operator $A(\e \pa_x)$ in \eqref{def Q} satisfies 
\begin{equation} \label{zero secco}
\Re \langle i \e^{-2} A(\e \pa_x) \ph , \ph \rangle_{H^s_\e} = 0
\end{equation}
because $\lm_j$ are all real.  
To estimate the term with $B_r(u, \e \pa_x)$ in \eqref{def Q}, 
we recall that 
\[
2 \Re \langle X \ph, \ph \rangle_{H^s_\e} 
= \langle (X + X^*) \Lm^s_\e \ph , \Lm^s_\e \ph \rangle_{L^2}
+ 2 \Re \langle [\Lm^s_\e, X] \ph, \Lm^s_\e \ph \rangle_{L^2}
\]
for any linear operator $X$, 
where $X^*$ is the adjoint of $X$ with respect to the $L^2$ scalar product
and $[ \, , \, ]$ is the commutator, 
whence 
\[
2 | \Re \langle X \ph, \ph \rangle_{H^s_\e} |
\leq \| X + X^* \|_{\mL(L^2, L^2)} \| \ph \|_{H^s_\e}^2 
+ 2 \| \, [\Lm^s_\e, X] \ph \|_{L^2} \| \ph \|_{H^s_\e}.
\]
By the Hermitian structure of $B_r(u, \e\pa_x)$, 
\begin{equation} \label{X+X*}
\| X + X^* \|_{\mL(L^2, L^2)} 
\lesssim \e^{-1} \| u \|_{L^\infty}^{p-1} \| u \|_{W^{1,\infty}_\e}, 
\quad \ 
X = \e^{-1} B_r(u, \e \pa_x),
\end{equation}
and, by \eqref{0503.6}, for $X = \e^{-1} B_r(u, \e \pa_x)$ one has
\[
\| \, [\Lm^s_\e, X] \ph \|_{L^2}
\lesssim_s \e^{-1} \| u \|_{L^\infty}^{p-1} ( \| u \|_{W^{1,\infty}_\e} \| \ph \|_{H^s_\e} 
+ \| u \|_{W^{[s]+2,\infty}_\e} \| \ph \|_{L^2}).
\]
Therefore 
\begin{align} 
|\Re \langle \e^{-1} B_r(u, \e \pa_x) \ph, \ph \rangle_{H^s_\e}| 
& \lesssim_s \e^{-1} \| u \|_{L^\infty}^{p-1} ( \| u \|_{W^{1,\infty}_\e} \| \ph \|_{H^s_\e} 
+ \| u \|_{W^{[s]+2,\infty}_\e} \| \ph \|_{L^2}) \| \ph \|_{H^s_\e},
\notag \\
|\Re \langle \e^{-1} B_r(u, \e \pa_x) \ph, \ph \rangle_{L^2}| 
& \lesssim \e^{-1} \| u \|_{L^\infty}^{p-1} \| u \|_{W^{1,\infty}_\e} \| \ph \|_{L^2}^2.
\label{est Br L2}
\end{align}
By \eqref{est G}-\eqref{est Br L2}, 
the solution $\ph$ of the linear equation $\pa_t \ph + Q(u)\ph = g_1$ 
(see \eqref{Cp.Q} and \eqref{def Q}) satisfies 
\begin{align}
\pa_t (\| \ph \|_{H^s_\e}^2)
& = 2 \Re \langle g_1 + \e^{-1} B_r(u, \e \pa_x) \ph - G(u) \ph , \ph \rangle_{H^s_\e} 
\notag \\ 
& \lesssim_s  \big\{ \| g_1 \|_{H^s_\e} 
+ \e^{-1} \| u \|_{L^\infty}^{p-1} 
( \| u \|_{W^{1,\infty}_\e} + \e^2 \| \pa_t u \|_{L^\infty} ) \| \ph \|_{H^s_\e} 
\notag \\ 
& \quad \ \ 
+ \e^{-1} \big( \| u \|_{L^\infty}^{p-1} 
( \| u \|_{W^{[s]+2,\infty}_\e} + \e^2 \| \pa_t u \|_{W^{[s]+1,\infty}_\e} )
\notag \\
& \quad \ \ 
+ \e^2 \| u \|_{L^\infty}^\nu \| u \|_{W^{[s]+1,\infty}_\e} \| \pa_t u \|_{L^\infty} \big)
\| \ph \|_{L^2} \big\} \| \ph \|_{H^s_\e},
\label{est ph Hs squarato}
\\
\pa_t (\| \ph \|_{L^2}^2)
& \lesssim 
\e^{-1} \| u \|_{L^\infty}^{p-1} 
(\| u \|_{W^{1,\infty}_\e} + \e^2 \| \pa_t u \|_{L^\infty}) \| \ph \|_{L^2}^2
+ \| g_1 \|_{L^2} \| \ph \|_{L^2}. 
\label{est ph L2 squarato}
\end{align}
If $u$ satisfies 
\begin{equation} \label{ansatz.marzo.1} 
\e^{-1} \| u \|_{L^\infty}^{p-1} 
(\| u \|_{W^{1,\infty}_\e} + \e^2 \| \pa_t u \|_{L^\infty}) \leq 1
\end{equation}
on the time interval $[0,T]$, 
then for $s \geq 0$ the solution $\ph$ of \eqref{Cp.Q} satisfies,
with the notation introduced in \eqref{def C Hs}, \eqref{def C Wm},
\begin{align} 
\| \ph \|_{C^0 L^2} 
& \lesssim \| g_1 \|_{C^0 L^2} + \| g_2 \|_{L^2},
\label{est ph L2}
\\
\| \ph \|_{C^0 H^s_\e} 
& \lesssim_s 
\| g_1 \|_{C^0 H^s_\e} + \| g_2 \|_{H^s_\e}
+ \e^{-1} \big( \| u \|_{C^0 L^\infty}^{p-1} \| u \|_{C^1_\e W^{[s]+3}_\e}
\notag \\ & \quad \ 
+ \| u \|_{C^0 L^\infty}^\nu \| u \|_{C^0 W^{[s]+1}_\e} \| u \|_{C^1_\e W^2_\e} \big) 
(\| g_1 \|_{C^0 L^2} + \| g_2 \|_{L^2})
\label{est ph Hs}
\end{align}
(first use \eqref{est ph L2 squarato}, \eqref{ansatz.marzo.1} and Gronwall to get \eqref{est ph L2}, 
then insert \eqref{est ph L2} into \eqref{est ph Hs squarato} and use Gronwall again).

By definitions \eqref{change h ph}, \eqref{def g1 g2}
and estimates \eqref{pasubio.11}, \eqref{pasubio.-1-1}, \eqref{pasubio.12},  
we deduce that the solution $h$ of the linear Cauchy problem \eqref{Cp.lin} 
satisfies the same estimates \eqref{est ph L2}, \eqref{est ph Hs} as $\ph$ 
with $f_1, f_2$ in place of $g_1, g_2$,
namely, for all $s \geq 0$, 
\begin{align} \label{est h L2}
\| h \|_{C^0 L^2} 
& \lesssim \| f_1 \|_{C^0 L^2} + \| f_2 \|_{L^2},
\\
\| h \|_{C^0 H^s_\e} 
& \lesssim_s 
\| f_1 \|_{C^0 H^s_\e} + \| f_2 \|_{H^s_\e}
+ \e^{-1} \big( \| u \|_{C^0 L^\infty}^{p-1} \| u \|_{C^1_\e W^{[s]+3}_\e}
\notag \\ & \quad \ 
+ \| u \|_{C^0 L^\infty}^\nu \| u \|_{C^0 W^{[s]+1}_\e} \| u \|_{C^1_\e W^2_\e} \big) 
(\| f_1 \|_{C^0 L^2} + \| f_2 \|_{L^2}).
\label{est h Hs}
\end{align}
From the equation $\pa_t h + P'(u)h = f_1$ one has, 
for all $s$ real, 
\begin{equation} \label{1603.1}
\| \pa_t h \|_{H^s_\e} \leq \| f_1 \|_{H^s_\e} + \| P'(u)h \|_{H^s_\e}.
\end{equation}
By \eqref{B -1}, \eqref{est R0 L2}, \eqref{ansatz.marzo.1}, 
for $-1 \leq s \leq 0$ one has
\begin{equation} \label{1603.2}
\| P'(u)h \|_{H^s_\e} 
\lesssim \e^{-2} \| h \|_{H^{s+2}_\e},
\end{equation}
and, by \eqref{B generica}, \eqref{est R0 Hs}, \eqref{ansatz.marzo.1}, 
for $s \geq 0$ one has 
\begin{equation} \label{1603.3}
\| P'(u)h \|_{H^s_\e} 
\lesssim_s \e^{-2} \| h \|_{H^{s+2}_\e} 
+ \e^{-1} \| u \|_{L^\infty}^{p-1} \| u \|_{W^{[s]+2,\infty}_\e} \| h \|_{L^2}.
\end{equation}
Hence, by \eqref{est h L2}-\eqref{1603.3},
for all $s \geq -1$ one has 
\begin{align*}
\e^2 \| \pa_t h \|_{H^s_\e} 
& \lesssim_s \| f_1 \|_{C^0 H^{s+2}_\e} + \| f_2 \|_{H^{s+2}_\e}
+ \e^{-1} \big( \| u \|_{C^0 L^\infty}^{p-1} \| u \|_{C^1_\e W^{[s]+5}_\e}
\notag \\ & \quad \ 
+ \| u \|_{C^0 L^\infty}^\nu \| u \|_{C^0 W^{[s]+3}_\e} \| u \|_{C^1_\e W^2_\e} \big) 
(\| f_1 \|_{C^0 L^2} + \| f_2 \|_{L^2}).
\end{align*}
Thus, recalling definition \eqref{def C Hs}, $h$ satisfies, for all $s \geq 1$,
\begin{align}
\| h \|_{C^1_\e H^s_\e}
& \lesssim_s \| f_1 \|_{C^0 H^s_\e} + \| f_2 \|_{H^s_\e}
+ \e^{-1} \big( \| u \|_{C^0 L^\infty}^{p-1} \| u \|_{C^1_\e W^{[s]+3}_\e}
\notag \\ & \quad \ 
+ \| u \|_{C^0 L^\infty}^\nu \| u \|_{C^0 W^{[s]+1}_\e} \| u \|_{C^1_\e W^2_\e} \big) 
(\| f_1 \|_{C^0 L^2} + \| f_2 \|_{L^2}).
\label{stima inv generica}
\end{align}
In conclusion, we have proved the following result.

\begin{lemma}[Right inverse of the linearized problem] 
\label{lemma:inv generale}
Let $s \geq 1$ be real, 
and let $u$ belong to 
$C([0,T], W^{[s]+3,\infty}(\R^d)) \cap C^1([0,T], W^{[s]+1,\infty}(\R^d))$, 
with \eqref{ansatz.marzo.1} and \eqref{ball for M}. 
Then for all $f_1 \in C([0,T], H^s(\R^d))$, 
all $f_2 \in H^s(\R^d)$, the linear Cauchy problem \eqref{Cp.lin}
has a (unique) solution $h$, which satisfies \eqref{stima inv generica}.
\end{lemma}

\textbf{Estimate for the second derivative.} 
By \eqref{diff.str} and \eqref{def P u0}, the operator 
\begin{align*} 
P''(u)[h_1, h_2]
& = - \e^{-1} (\pa_u B)(u, \e \pa_x)[h_1] h_2 
- \e^{-1} (\pa_u B)(u, \e \pa_x)[h_2] h_1
\\ & \quad \ 
- \e^{-1} (\pa_{uu} B)(u, \e \pa_x)[h_1, h_2] u 
\end{align*}
is the sum of terms of the form 
\begin{equation}\label{g'' informal}
\e^{-1} g'(u) h_1 \, \e \pa_x h_2
+ \e^{-1} g'(u) h_2 \, \e \pa_x h_1
+ \e^{-1} g''(u) h_1 h_2 \, \e \pa_x u,
\end{equation}
where $g(u)$ is a vector of components 
$b_{\ell j k}(u)$ or $c_{\ell jk}(u)$. 
By \eqref{order p}, $g(u) = O(|u|^p)$ 
with $p \geq 1$ integer.
For $p \geq 3$, 
by \eqref{FS.comp} one has
for all $u$ in the ball $\| u \|_{L^\infty} \leq 1$, 
for all $s \geq 0$,
\begin{alignat}{2} 
\| g'(u) \|_{L^\infty} 
& \lesssim \| u \|_{L^\infty}^{p-1},
& \qquad 
\| g'(u) \|_{H^s_\e} 
& \lesssim_s \| u \|_{H^s_\e} \| u \|_{L^\infty}^{p-2},
\label{1301.2} 
\\
\| g''(u) \|_{L^\infty} 
& \lesssim \| u \|_{L^\infty}^{p-2},
& \qquad 
\| g''(u) \|_{H^s_\e} 
& \lesssim_s \| u \|_{H^s_\e} \| u \|_{L^\infty}^{p-3}.
\label{1301.3} 
\end{alignat}
For $p=2$, $g(u) = g_2(u) + \tilde g(u)$ 
where $g_2(u)$ is homogeneous of degree 2 in $u$ 
and $\tilde g(u) = O(|u|^3)$ 
(we do not distinguish whether $\tilde g$ is of order 3 or higher).
Thus $\tilde g(u)$ satisfies \eqref{1301.2}-\eqref{1301.3} with 3 in place of $p$, 
and $g_2$ satisfies \eqref{1301.2} with 2 in place of $p$, 
while $g_2''(u)$ is a constant, independent of $u$.   
For $p=1$, one has $g(u) = g_1(u) + g_2(u) + \tilde g(u)$ 
where $g_1(u)$ is linear in $u$ and $g_2, \tilde g$ are as above. 
Thus $g_1'(u)$ is a constant, independent of $u$, and $g_1''(u)=0$. 

By \eqref{g'' informal}, \eqref{FS.prod} and \eqref{FS.comp}, 
for all $u$ in the ball $\| u \|_{L^\infty} \leq 1$, 
for all real $s \geq 0$, all integer $p \geq 1$, one has 
\begin{align}
\| P''(u)[h_1,h_2] \|_{H^s_\e} 
& \lesssim_s \e^{-1} \| u \|_{L^\infty}^{p-1}  
(\| h_1 \|_{H^{s+1}_\e} \| h_2 \|_{L^\infty}
+ \| h_1 \|_{H^s_\e} \| h_2 \|_{W^{1,\infty}_\e} 
\notag \\ & \hspace{60pt}
+ \| h_1 \|_{W^{1,\infty}_\e} \| h_2 \|_{H^s_\e} 
+ \| h_1 \|_{L^\infty} \| h_2 \|_{H^{s+1}_\e})  
\notag \\ & \quad \
+ \e^{-1} \| u \|_{L^\infty}^\nu \| u \|_{W^{1, \infty}_\e} 
( \| h_1 \|_{H^s_\e} \| h_2 \|_{L^\infty} 
+ \| h_1 \|_{L^\infty} \| h_2 \|_{H^s_\e} )
\notag \\ & \quad \
+ \e^{-1} \| u \|_{L^\infty}^\nu \| u \|_{H^s_\e}
( \| h_1 \|_{W^{1, \infty}_\e} \| h_2 \|_{L^\infty} 
+ \| h_1 \|_{L^\infty} \| h_2 \|_{W^{1, \infty}_\e}) 
\notag \\ & \quad \
+ \e^{-1} (\| u \|_{L^\infty}^\nu \| u \|_{H^{s+1}_\e} 
+ \| u \|_{L^\infty}^{\nu_3} \| u \|_{W^{1, \infty}_\e} \| u \|_{H^s_\e}) 
\| h_1 \|_{L^\infty} \| h_2 \|_{L^\infty},
\label{1103.1}
\end{align}
where $\nu = \max \{ p-2, 0 \}$ has been defined in \eqref{def nu},
and $\nu_3 := \max\{ p-3, 0 \}$. 

\bigskip

\textbf{Estimates in $H^s_\e$ spaces only.} 
For the result in the concentrating case, 
it is convenient to work directly in $H^s_\e$ class, 
avoiding the $W^{m,\infty}_\e$ spaces. 
Thus, by \eqref{0701.B}, one has 
\begin{equation} \label{NS.B}
\| B(u,\e \pa_x) h \|_{H^s_\e} 
\leq \e^{-pd/2} (C_{s_0} \| u \|_{H^{s_0}_\e}^p \| h \|_{H^{s+1}_\e} 
+ C_s \| u \|_{H^{s_0}_\e}^{p-1} \| u \|_{H^s_\e} \| h \|_{H^{s_0+1}_\e})
\end{equation}
for all $s \geq s_0 > d/2$, all $u$ in the ball 
\begin{equation} \label{NS.ball.1}
C_{s_0} \e^{-d/2} \| u \|_{H^{s_0}_\e} \leq 1,
\end{equation}
so that $\| u \|_{L^\infty} \leq 1$. 
By \eqref{NS.B} and \eqref{def M}, 
\begin{equation} \label{NS.M.1}
\| \op_\e(M) h \|_{H^s_\e} 
\leq \e^{-pd/2} (C_{s_0} \| u \|_{H^{s_0}_\e}^p \| h \|_{H^{s-1}_\e} 
+ C_s \| u \|_{H^{s_0}_\e}^{p-1} \| u \|_{H^s_\e} \| h \|_{H^{s_0-1}_\e})
\end{equation}
for $s \geq s_0 > d/2$, $u$ in the ball \eqref{NS.ball.1}. 
Thus there exists $\rho_3 > 0$, independent of $\e$, such that 
for $u$ in the ball
\begin{equation} \label{NS.ball.2}
\e^{1-pd/2} \| u \|_{H^{s_0}_\e}^p \leq \rho_3,
\end{equation}
one has  
\begin{equation} \label{NS.M.2}
\| \e \op_\e(M) h \|_{H^{s_0}_\e} 
\leq C_{s_0} \e^{1-pd/2} \| u \|_{H^{s_0}_\e}^p \| h \|_{H^{s_0-1}_\e} 
\leq \tfrac12 \| h \|_{H^{s_0-1}_\e}.
\end{equation}
Therefore, by Neumann series, $I + \e \op_\e(M)$ is invertible in $H^{s_0}(\R^d)$, 
and 
\begin{equation} \label{NS.M.3}
\| (I + \e \op_\e(M))^{-1} h \|_{H^s_\e} 
\leq C_{s_0} \| h \|_{H^s_\e} 
+ C_s \e^{1-pd/2} \| u \|_{H^{s_0}_\e}^{p-1} \| u \|_{H^s_\e} \| h \|_{H^{s_0-1}_\e}
\end{equation}
for $s \geq s_0 > d/2$ and $u$ in the ball \eqref{NS.ball.2}.
For $u$ in the ball \eqref{NS.ball.2}, 
for $s \geq s_0 > d/2$, 
we deduce the following estimates:
\begin{multline} \label{NS.IMMBr}
\| (I + \e \op_\e(M))^{-1} \e \op_\e(M) \e^{-1} B_r(u, \e \pa_x) \ph \|_{H^s_\e} 
\\
\lesssim_s \e^{-pd/2} \| u \|_{H^{s_0}_\e}^{p-1} 
( \| u \|_{H^{s_0+1}_\e} \| \ph \|_{H^s_\e} 
+  \| u \|_{H^s_\e} \| \ph \|_{H^{s_0}_\e})
\end{multline}
(to prove \eqref{NS.IMMBr}, we have used \eqref{prod pax Hs}),
\begin{equation} \label{NS.Bf}
\| \e^{-1} B_{lf}(u, \e \pa_x) \ph \|_{H^s_\e}
\lesssim_s \e^{-1-pd/2} \| u \|_{H^{s_0}_\e}^{p-1} \| u \|_{H^s_\e} \| \ph \|_{L^2},
\end{equation}
\begin{multline} 
\label{NS.pat M}
\| \e \op_\e(\pa_t M) \ph \|_{H^s_\e} 
\lesssim_s \e^{1-pd/2} 
\{ \| u \|_{H^{s_0}_\e}^{p-1} \| \pa_t u \|_{H^{s_0}_\e} \| \ph \|_{H^{s-1}_\e}
\\
+ (\| u \|_{H^{s_0}_\e}^{p-1} \| \pa_t u \|_{H^s_\e} 
+ \| u \|_{H^{s_0}_\e}^{\nu} \| u \|_{H^s_\e} \| \pa_t u \|_{H^{s_0}_\e}) 
\| \ph \|_{H^{s_0-1}_\e} \}
\end{multline}
with $\nu$ defined in \eqref{def nu}, 
and
\begin{equation} \label{NS.R0}
\| R_0(u) \ph \|_{H^s_\e} 
\lesssim_s \e^{-1-pd/2} \| u \|_{H^{s_0}_\e}^{p-1} 
(\| u \|_{H^{s_0+1}_\e} \| \ph \|_{H^s_\e}
+ \| u \|_{H^{s+1}_\e} \| \ph \|_{H^{s_0}_\e}).
\end{equation}
Hence 
\begin{align} \label{NS.G}
\| G(u) \ph \|_{H^s_\e} 
& \lesssim_s \e^{-1-pd/2} \| u \|_{H^{s_0}_\e}^{p-1} 
(\| u \|_{H^{s_0+1}_\e} + \e^2 \| \pa_t u \|_{H^{s_0}_\e}) \| \ph \|_{H^s_\e}
\notag \\ & \quad \ 
+ \e^{-1-pd/2} \{ \| u \|_{H^{s_0}_\e}^{p-1} 
(\| u \|_{H^{s+1}_\e} + \e^2 \| \pa_t u \|_{H^s_\e})
\notag \\ & \quad \ 
+ \e^2 \| u \|_{H^{s_0}_\e}^{\nu} \| u \|_{H^s_\e} \| \pa_t u \|_{H^{s_0}_\e} \} 
\| \ph \|_{H^{s_0}_\e}
\end{align}
for all $s \geq s_0$.
By \eqref{X+X*} and \eqref{NS.commu.B}, 
for $X = \e^{-1} B_r(u, \e \pa_x)$, for $s \geq s_0$, 
one has 
\begin{align}
\| X + X^* \|_{\mL(L^2,L^2)} 
& \lesssim \e^{-1-pd/2} \| u \|_{H^{s_0}_\e}^{p-1} \| u \|_{H^{s_0+1}_\e},
\notag \\
\| \, [ \Lm^s_\e, X] \ph \|_{L^2} 
& \lesssim_s \e^{-1-pd/2} \| u \|_{H^{s_0}_\e}^{p-1} 
( \| u \|_{H^{s_0+1}_\e} \| \ph \|_{H^s_\e} 
+ \| u \|_{H^{s+1}_\e} \| \ph \|_{H^{s_0}_\e}),
\notag \\
| \Re \langle X \ph, \ph \rangle_{H^s_\e} | 
& \lesssim_s \e^{-1-pd/2} \| u \|_{H^{s_0}_\e}^{p-1} 
( \| u \|_{H^{s_0+1}_\e} \| \ph \|_{H^s_\e} 
+ \| u \|_{H^{s+1}_\e} \| \ph \|_{H^{s_0}_\e}) \| \ph \|_{H^s_\e}.  
\label{NS.ReX}
\end{align}
By \eqref{zero secco}, \eqref{NS.G} and \eqref{NS.ReX}, 
we get energy estimates for $\ph$:  
for $u$ in the ball
\begin{equation} \label{NS.ball.3}
\e^{-1-pd/2} \| u \|_{C^1_\e H^{s_0+2}_\e}^p \leq 1,
\end{equation}
the solution $\ph$ of the linear Cauchy problem \eqref{Cp.Q} 
satisfies 
\begin{align} \label{NS.energy.bassa}
\| \ph \|_{C^0 H^{s_0}_\e} 
& \lesssim \| g_1 \|_{C^0 H^{s_0}_\e} + \| g_2 \|_{H^{s_0}_\e},
\\
\| \ph \|_{C^0 H^s_\e} 
& \lesssim_s \| g_1 \|_{C^0 H^s_\e} + \| g_2 \|_{H^s_\e} 
\notag \\ & \quad \ 
+ \e^{-1-pd/2} \| u \|_{C^1_\e H^{s_0+2}_\e}^{p-1} \| u \|_{C^1_\e H^{s+2}_\e}
(\| g_1 \|_{C^0 H^{s_0}_\e} + \| g_2 \|_{H^{s_0}_\e})
\label{NS.energy.alta}
\end{align}
for all $s \geq s_0$. 
Hence, following the same argument as above, 
the solution $h$ of the Cauchy problem \eqref{Cp.lin} 
satisfies, for $s \geq s_0+2$, 
\begin{align} 
\| h \|_{C^1_\e H^s_\e} 
& \lesssim_s \| f_1 \|_{C^0 H^s_\e} + \| f_2 \|_{H^s_\e} 
\notag \\ & \quad \ 
+ \e^{-1-pd/2} \| u \|_{C^1_\e H^{s_0+2}_\e}^{p-1} \| u \|_{C^1_\e H^{s+2}_\e}
(\| f_1 \|_{C^0 H^{s_0}_\e} + \| f_2 \|_{H^{s_0}_\e}).
\label{NS.energy.h}
\end{align}
We have obtained the following inversion for the linear problem. 

\begin{lemma}
\label{lemma:inv senza W}
Let $s_0 > d/2$, $s \geq s_0 + 2$,
and $u \in C([0,T], H^{s+2}(\R^d)) \cap C^1([0,T], H^s(\R^d))$, 
with \eqref{NS.ball.1}, \eqref{NS.ball.2} and \eqref{NS.ball.3}.
Then for all $f_1 \in C([0,T],H^s(\R^d))$, all $f_2 \in H^s(\R^d)$, 
the linear Cauchy problem \eqref{Cp.lin} has a (unique) solution $h$,
which satisfies \eqref{NS.energy.h}.
\end{lemma}

Also, by \eqref{1103.1} and \eqref{NS.ball.1}, for $s \geq s_0$, 
\begin{align}
\| P''(u)[h_1, h_2] \|_{H^s_\e} 
& \lesssim_s \e^{-1-pd/2} \| u \|_{H^{s_0}_\e}^{p-1} 
(\| h_1 \|_{H^{s+1}_\e} \| h_2 \|_{H^{s_0}_\e} 
+ \| h_1 \|_{H^{s_0}_\e} \| h_2 \|_{H^{s+1}_\e})
\notag \\ & \quad \ 
+ \e^{-1-(\nu+2)d/2} \| u \|_{H^{s_0}_\e}^{\nu} 
\| u \|_{H^{s+1}_\e} \| h_1 \|_{H^{s_0}_\e} \| h_2 \|_{H^{s_0}_\e}.
\label{NS.P''}
\end{align}

\section{Proof of Theorem \ref{thm:no loss}}
\label{sec:proof no loss}

For $\anm \geq 0$ real, let
\begin{align} 
E_\anm 
& := C([0,T],H^{s_0 + \anm}(\R^d)) 
\cap C^1([0,T],H^{s_0 + \anm - 2}(\R^d)), 
\label{def Ea.1} 
\\
F_\anm 
& := C([0,T],H^{s_0+\anm}(\R^d)) \times H^{s_0+\anm}(\R^d),
\label{def Fa.1}
\end{align}
and, recalling the notation in \eqref{def C Hs}, 
define
\begin{equation}
\| u \|_{E_\anm} 
:= \| u \|_{C^1_\e H^{s_0+\anm}_\e},
\qquad 
\| f \|_{F_\anm} 
= \| (f_1, f_2) \|_{F_\anm}
:= \| f_1 \|_{C^0 H^{s_0+\anm}_\e} + \| f_2 \|_{H^{s_0+\anm}_\e}.
\end{equation}
Define the smoothing operators $S_j$, $j \in \N$, 
as the ``semi-classical'' crude Fourier truncations 
\begin{equation} \label{def Sj}
S_j u(x) := (2\pi)^{-d/2} \int_{\e |\xi| \leq 2^j} \hat u(\xi) e^{i \xi \cdot x} \, d\xi,
\end{equation}
which satisfy all \eqref{S1}-\eqref{2705.4} with constants independent of $\e$. 
Define
\begin{equation} \label{def Phi.1}
\Phi(u) := \begin{pmatrix} 
\pa_t u + P(u) \\ 
u(0) \end{pmatrix},
\end{equation}
where $P(u)$ is defined in \eqref{def P u0}. 
For $\| u \|_{E_2} \leq 1$, 
the second derivative of $\Phi$ satisfies \eqref{NS.P''}, 
which gives, for all $\anm \geq 0$, 
\begin{align}
\| \Phi''(u)[h_1, h_2] \|_{F_\anm} 
& \lesssim_s \e^{-1-pd/2} \| u \|_{E_0}^{p-1} 
(\| h_1 \|_{E_{\anm+1}} \| h_2 \|_{E_0} 
+ \| h_1 \|_{E_0} \| h_2 \|_{E_{\anm+1}})
\notag \\ & \quad \ 
+ \e^{-1-(\nu+2)d/2} \| u \|_{E_0}^{\nu} 
\| u \|_{E_{\anm+1}} \| h_1 \|_{E_0} \| h_2 \|_{E_0}.
\label{NS.P''.1}
\end{align}
For $u$ in the ball 
\begin{equation} \label{NS.ball.4}
\| u \|_{E_2} \leq \e^q, \quad \ q := \frac{1}{p} + \frac{d}{2},
\end{equation}
the conditions \eqref{NS.ball.1}, 
\eqref{NS.ball.2}, 
\eqref{NS.ball.3} 
are all satisfied for $\e$ sufficiently small --- more precisely,
for $\e \in (0,\e_0]$, 
where $\e_0 := \min \{ 1, C_{s_0}^{-p}, \rho_3^{1/2} \}$, 
and $C_{s_0}, \rho_3$ are the constants in \eqref{NS.ball.1}, \eqref{NS.ball.2}, 
independent of $\e$. 
Then, for $u$ in the ball \eqref{NS.ball.4}, 
Lemma \ref{lemma:inv senza W} defines a right inverse $\Psi(u)$ 
of the linearized operator $\Phi'(u)$ 
(namely $h = \Psi(u) f$ solves the linear Cauchy problem 
$\Phi'(u) h = f$, which is \eqref{Cp.lin}),  
with bound \eqref{NS.energy.h}, which is 
\begin{equation} 
\| \Psi(u) f \|_{E_\anm} 
\lesssim_s \| f \|_{F_\anm}
+ \e^{-1-pd/2} \| u \|_{E_2}^{p-1} \| u \|_{E_{\anm+2}} \| f \|_{F_0},
\quad \ \anm \geq 2.
\label{NS.Psi.1}
\end{equation}
To reach the best radius for the initial data
(see Remark \ref{rem:best rescaling} and Remark \ref{rem:specific}), 
we introduce the rescaled norm
\begin{equation} \label{NS.rescaled.1}
\| u \|_{\mE_\anm} := \e^{-q} \| u \|_{E_\anm}.
\end{equation}
Thus \eqref{NS.ball.4} becomes 
\begin{equation} \label{NS.ball.5}
\| u \|_{\mE_2} \leq 1.
\end{equation}
By \eqref{NS.P''.1} and \eqref{NS.Psi.1}, 
for all $u$ in the unit ball \eqref{NS.ball.5}
one has
\begin{align}
\| \Phi''(u)[h_1, h_2] \|_{F_\anm} 
& \lesssim_s \e^q 
(\| h_1 \|_{\mE_{\anm+1}} \| h_2 \|_{\mE_0} 
+ \| h_1 \|_{\mE_0} \| h_2 \|_{\mE_{\anm+1}}
\notag \\ & \quad \ 
+ \| u \|_{\mE_{\anm+1}} \| h_1 \|_{\mE_0} \| h_2 \|_{\mE_0})
\label{NS.P''.2}
\end{align}
for $\anm \geq 0$, 
because $-1-(\nu+2)d/2 + q(\nu+3) \geq q$ (recall that $\nu = \max \{ p-2,0 \}$), 
and 
\begin{equation} \label{NS.Psi.2}
\| \Psi(u) f \|_{\mE_\anm} 
\lesssim_s \e^{-q} (\| f \|_{F_\anm} + \| u \|_{\mE_{\anm+2}} \| f \|_{F_0})
\end{equation}
for $\anm \geq 2$. 
Hence $\Phi$ satisfies the assumptions of Theorem \ref{thm:NMH} with
\begin{gather}
\anm_0 = 0, \quad 
\mu = \anm_1 = 2, \quad 
\b = \a > 4, \quad 
\anm_2 > 2\b - 2, \quad 
U = \{ u \in E_2 : \| u \|_{\mE_2} \leq 1 \}, 
\notag \\
\d_1 = 1, \quad 
M_1(\anm) = M_2(\anm) = C_\anm \e^q, \quad 
L_1(\anm) = L_2(\anm) = C_\anm \e^{-q}, \quad 
M_3(\anm) = L_3(\anm) = 0.
\label{LMd.1}
\end{gather}
For any function $u_0 = u_0(x) \in H^{s_0+\b}(\R^d)$, 
the pair $g = (0, u_0)\in F_\b$  
trivially satisfies the first inequality in \eqref{2705.1} with $A=1$
(in fact, the inequality is an identity), 
because $g$ does not depend on the time variable. 

Hence, by Theorem \ref{thm:NMH}, 
if $\| g \|_{F_\b} \leq \d$,  
with $\d = C \e^q$ given by \eqref{qui.02}, 
there exists $u \in E_\a$ such that 
$\Phi(u) = \Phi(0) + g = g$. 
This means that we have solved 
the nonlinear Cauchy problem \eqref{Cp.general}, 
i.e.\ $\Phi(u) = (0, u_0)$,  
on the time interval $[0,T]$ 
for all initial data $u_0$ in the ball 
\begin{equation} \label{NS.ball.6}
\| u_0 \|_{H^{s_0+\b}_\e} \leq \d = C \e^q,
\end{equation}
for all $\e \in (0, \e_0]$. 
By \eqref{qui.01}, the solution $u$ satisfies 
\[
\| u \|_{\mE_\a} 
\leq C \e^{-q} \| g \|_{F_\b},
\quad \text{i.e.} \ \ 
\| u \|_{C^1_\e H^{s_0+\b}_\e} \leq C \| u_0 \|_{H^{s_0+\b}_\e}.
\]
The higher regularity part of Theorem \ref{thm:no loss} 
is also deduced from Theorem \ref{thm:NMH}.  

For data $u_0$ of the form $u_0(x) = \e^\s (\aff_\e(x), \overline{\aff_\e(x)} )$ 
(see \eqref{def P u0}), where $\aff_\e$ is defined in \eqref{1012.3}, 
one has 
\[
\| u_0 \|_{H^s_\e} 
= \e^\s \| \aff_\e \|_{H^s_\e} 
\lesssim_s \e^{\s + \sigmaa} \| \aff \|_{H^s},
\]
see \eqref{Sob sigma a}, \eqref{def mT eps}, 
where $\sigmaa = d/2$ in the concentrating case, 
and $\sigmaa = 0$ in the fast oscillating case.
Hence $u_0$ belongs to the ball \eqref{NS.ball.6} 
for all $\e$ sufficiently small if 
\[
\| u_0 \|_{H^{s_0+\b}_\e} 
\leq C_{s_0+\b} \e^{\s + \sigmaa} \| \aff \|_{H^{s_0+\b}} 
\leq \d = C \e^q.
\]
For $\| \aff \|_{H^{s_0+\b}} \leq 1$, this holds for $\s + \sigmaa > q$, 
namely 
\[
\s > \frac{1}{p} + \frac{d}{2} - \sigmaa.
\]

Finally, given $s_1 > d/2 + 4$, we define $\g := s_1 - (d/2 + 4)$, 
$s_0 := d/2 + \g/2$, $\b := 4 + \g/2$, 
so that $s_0 > d/2$, $\b > 4$, and $s_1 = s_0 + \b$.
This concludes the proof of Theorem \ref{thm:no loss}.

\begin{remark}[\emph{Confirmation of the heuristics discussion 
of Section \ref{sec:antani} in Theorem \ref{thm:no loss}}] 
\label{rem:confirm.1}
The radius $\d$ given by the Nash-Moser Theorem \ref{thm:NMH} 
is the minimum among $1/L$, $\d_1 / L$, $1 / (L^2 M)$;
here (see \eqref{LMd.1}) 
these three quantities are all of order $\e^q$. 
In particular, the ``quadratic condition'' $\d \leq 1 / (L^2 M)$,  
coming from the use of the second derivative $\Phi''(u)$ in the Nash-Moser iteration,
does not modify $\d$. 
This is a confirmation of the heuristic discussion 
of Section \ref{sec:antani}.
\end{remark}

\section{Free flow component decomposition}
\label{sec:L infty}

The ``shifted map'' trick used in \cite{TZ} and \cite{ES} 
consists in choosing the solution of the linear part of the PDE as a starting point 
for the Nash-Moser iteration.
The reason for which that trick works is that 
the free flow of functions of special structure \eqref{1012.3} 
satisfies better estimates in $L^\infty$ norm 
than the free flow of general Sobolev functions. 
This, combined with the power $p$ of the nonlinearity in the equation,
makes it possible to obtain solutions of larger size, 
which are the sum of a free flow and a correction of smaller size.

Here we use this property in a different way, 
splitting the problem into components of special structure \eqref{1012.3} 
and corrections, introducing non-isotropic norms to catch the different size effect.

For any function $\aff \in H^s(\R^d)$ we define $\mT_\e \aff$, 
$0 < \e \leq 1$, as
\begin{equation} \label{def mT eps}
(\mT_\e \aff)(x) := 
\begin{cases}
\aff(x/\e) 
& \text{(concentrating case)},
\\
e^{ix \cdot \xi_0/\e} \aff(x)
& \text{(oscillating case)},  
\end{cases}
\end{equation}
so that, in both cases, \eqref{1012.3} 
becomes $\aff_\e = \mT_\e \aff_0$. 
To deal with conjugate pairs, define 
\[
\mT_{\e,c} \aff := (\mT_\e \aff, \overline{\mT_\e \aff}),
\quad 
\mT_{\e,c}^{-1} (\bff, \overline{\bff}) := \mT_\e^{-1} \bff.
\]
Hence the initial datum $u_0$ defined in \eqref{def P u0} 
can be written as $u_0 = \e^\s \mT_{\e,c} \aff_0$. 

\begin{lemma} \label{lemma:mTeps} 
Let $\aff \in H^s(\R^d)$, $s \geq 0$. Then the Fourier transform of $\mT_\e \aff$ is
\begin{equation} \label{explicit Fourier}
\widehat{(\mT_\e \aff)}(\xi) = \e^d \hat \aff(\e \xi) 
\ \text{(concentrating)}, 
\qquad \quad 
\widehat{(\mT_\e \aff)}(\xi) = \hat \aff(\xi - \xi_0/\e) 
\  \text{(oscillating)},
\end{equation}
and one has 
\begin{equation} \label{Sob sigma a}
\| \mT_\e \aff \|_{H^s_\e} \leq \e^{\sigmaa} (2 \| \aff \|_{H^s} + C_s \| \aff \|_{L^2})
\end{equation}
where 
\begin{equation} \label{def sigma a}
\sigmaa = d/2 
\quad \text{(concentrating)}, 
\quad \qquad 
\sigmaa = 0
\quad \text{(oscillating)}.
\end{equation}
\end{lemma}

\begin{proof} Formula \eqref{explicit Fourier} is a direct calculation. 
Then, in the concentrating case, $\| \mT_\e \aff \|_{H^s_\e} = \e^{d/2} \| \aff \|_{H^s}$. 
In the oscillating case, using the change of variable $\xi - \xi_0/\e = \eta$
and applying \eqref{elem.1}, 
one has $\| \mT_\e \aff \|_{H^s_\e} \leq 2 \| \aff \|_{H^s} + C_s |\xi_0|^s \| \aff \|_{L^2}$. 
\end{proof}

Given any $y_0 \in H^s(\R^d)$, let $y = \mS y_0$ denote the solution 
of the linear Cauchy problem
\begin{equation} \label{free flow}
\begin{cases}
\pa_t y + i \e^{-2} A(\e \pa_x) y = 0, \\
y(0,x) = y_0(x),
\end{cases}
\end{equation}
so that $\mS$ is the free Schr\"odinger solution map.
For initial data of type $\mT_{\e,c} \aff$, the flow $\mS \mT_{\e,c} \aff$ 
has special properties, which are used in the proof of Theorem 4.6 in \cite{TZ}, 
that we recall in the following lemma. 

\begin{lemma} \label{lemma:free flow}
For all real $s \geq 0$, $s_0 > d/2$, all multi-indices $\a \in \N^d$, 
for all $t \in \R$ the solution 
\[
y = \mS \mT_{\e,c} \aff
\] 
of the linear Cauchy problem \eqref{free flow} 
with initial datum $y_0 = \mT_{\e,c} \aff$
satisfies 
\begin{align} 
\label{y infty}
\| y(t) \|_{L^\infty} 
& \leq C_{s_0} \| \aff \|_{H^{s_0}}, 
\\
\e^2 \| \pa_t y(t) \|_{L^\infty} 
& \leq C_{s_0} \| \aff \|_{H^{s_0+2}},
\label{pat y infty}
\\
\e^{|\a|} \| \pa_x^\a y(t) \|_{L^\infty} 
& \leq C_{|\a|,s_0} \| \aff \|_{H^{s_0+|\a|}},
\label{pax y infty}
\\
\| y(t) \|_{H^s_\e} 
& = \| \mT_\e \aff \|_{H^s_\e}.
\label{Sobolev y}
\end{align}
\end{lemma}

\begin{proof}
At each $t$ one has $|y(t,x)| \lesssim \| \hat y(t,\cdot) \|_{L^1}$ 
by inverse Fourier formula, 
and $|\hat y(t,\xi)| = |\hat y(0,\xi)| = |\hat{(\mT_{\e,c} \aff)}|$ for all $t,\xi$ 
because $y$ solves \eqref{free flow}. 
By \eqref{explicit Fourier}, one has 
$\| \hat{(\mT_\e \aff)} \|_{L^1} = \| \hat \aff \|_{L^1}$ 
in both cases. 
This proves \eqref{y infty} because, by H\"older's inequality, 
$\| \hat \aff \|_{L^1} \lesssim_{s_0} \| \aff \|_{H^{s_0}}$.

To prove \eqref{pat y infty} we use the equation in \eqref{free flow}
recalling that $\e^{-2} A(\e \pa_x) = A(\pa_x)$. 
Proceeding as above, we get 
$|\pa_t y(t,x)| \lesssim \int |\xi|^2 |\hat{(\mT_\e \aff)}(\xi)| \, d\xi$, 
then we use \eqref{explicit Fourier} to conclude.
Similarly, \eqref{pax y infty} follows from 
$|\pa_x^\a y(t,x)| \lesssim \int |\xi|^{|\a|} |\hat{(\mT_\e \aff)}(\xi)| \, d\xi$. 
Finally, \eqref{Sobolev y} is trivial.
\end{proof}

We look for a solution of the Cauchy problem \eqref{Cp.general} 
by decomposing the unknown $u$ into the sum of the 
solution of the free Schr\"odinger equation with initial datum 
$u_0$ of the form \eqref{def P u0} 
and a ``correction'' $\tilde u(t,x)$ of smaller size. 

For any pair $(\aff, \tilde u)$ where $\aff = \aff(x) \in H^s(\R^d)$
and $\tilde u = \tilde u(t,x) \in C^0([0,T],H^s(\R^d)) \cap C^1([0,T], H^{s-2}(\R^d))$ 
with $\tilde u(0,x) = 0$, 
we define 
\begin{equation} \label{def Phi tilde}
\tilde \Phi(\aff, \tilde u) := 
\begin{pmatrix}
\pa_t u + P(u) \\ 
\aff 
\end{pmatrix}
\quad \text{where} \ u = \e^\s \mS \mT_{\e,c} \aff + \tilde u.
\end{equation}
At time $t=0$ the function $u$ in \eqref{def Phi tilde}
satisfies $u(0) = \e^\s \mT_{\e,c} \aff$.
Hence the Cauchy problem \eqref{Cp.general} becomes 
\begin{equation} \label{1703.1}
\tilde \Phi(\aff,\tilde u) = (0, \aff_0).
\end{equation}
We solve \eqref{1703.1} by applying our Nash-Moser-H\"ormander theorem; 
therefore we have to construct a right inverse for the linearized operator 
and to estimate the second derivative. 
We only have to adapt the general analysis of Section \ref{sec:general} 
to functions $u$ of the form \eqref{def Phi tilde}. 

\medskip

\textbf{Right inverse of the linearized operator.}
The differential of $\tilde \Phi$ at the point $(\aff, \tilde u)$ 
in the direction $(\bff, \tilde h)$ is 
\begin{equation} \label{def u h}
\tilde \Phi'(\aff, \tilde u) (\bff, \tilde h) 
= \begin{pmatrix} 
\pa_t h + P'(u) h \\ \bff 
\end{pmatrix} 
\quad \text{where} \ \ 
u = \e^\s \mS \mT_{\e,c} \aff + \tilde u, 
\ \ 
h = \e^\s \mS \mT_{\e,c} \bff + \tilde h
\end{equation}
and $\tilde u(0) = 0$, $\tilde h(0) = 0$. 
Given $(\aff, \tilde u)$ and $g = (g_1, g_2)$, with $g_1 = g_1(t,x)$ and $g_2 = g_2(x)$, 
the right inversion problem for the linearized operator $\tilde \Phi'(\aff, \tilde u)$ 
consists in finding $(\bff, \tilde h)$ such that
\begin{equation} \label{lin.g}
\tilde \Phi'(\aff, \tilde u) (\bff, \tilde h) = g, 
\quad \ \text{i.e.} \ \ 
\begin{cases}
\pa_t h + P'(u) h = g_1,
\\
\bffb = g_2
\end{cases}
\end{equation}
with $u,h$ as in \eqref{def u h}.
Since the free flow  
$\e^\s \mS \mT_{\e,c} \bff = \e^\s \mS \mT_{\e,c} g_2$ solves \eqref{free flow}, 
and $\tilde h(0) = 0$ by construction,  
\eqref{lin.g} is equivalent to the following problem for $\tilde h$: 
\begin{equation} \label{lin.with}
\begin{cases} 
\pa_t \tilde h + P'(u) \tilde h 
= g_1 + \e^{-1} B(u, \e \pa_x) \e^\s \mS \mT_{\e,c} g_2
- R_0(u) \e^\s \mS \mT_{\e,c} g_2, \\
\tilde h(0) = 0,
\end{cases}
\end{equation}
namely $\tilde h$ has to solve the linear Cauchy problem \eqref{Cp.lin} 
with 
\begin{equation} \label{fg spar}
f_1 = g_1 + \e^{-1} B(u, \e \pa_x) \e^\s \mS \mT_{\e,c} g_2
- R_0(u) \e^\s \mS \mT_{\e,c} g_2, \quad \ 
f_2 = 0.
\end{equation}
The solution of \eqref{Cp.lin} is estimated in Lemma \ref{lemma:inv generale};
to apply that lemma, now we check that $u$ satisfies its hypotheses.
By Lemma \ref{lemma:free flow}, 
\eqref{FS.emb}, 
\eqref{FS.der}, 
\eqref{def C Hs},  
\eqref{def C Wm}, 
the function $u = \e^\s \mS \mT_{\e,c} \aff + \tilde u$ satisfies 
\begin{align}
\| u \|_{C^1_\e W^m_\e} 
& \lesssim_{s_0,m} \e^\s \| \aff \|_{H^{s_0+m}} 
+ \e^{-d/2} \| \tilde u \|_{C^1_\e H^{s_0+m}_\e}, 
\quad \ m \in \N.
\label{stima u infty}
\end{align}
For all $s$, let
\begin{equation} \label{def Xs}
\| (\aff, \tilde u) \|_{X^s} := \e^\s \| \aff \|_{H^s} 
+ \e^{-d/2} \| \tilde u \|_{C^1_\e H^s_\e}.
\end{equation}
By \eqref{stima u infty}, one has, in particular,
\begin{equation} \label{u bassa}
\| u \|_{W^{2,\infty}_\e} + \e^2 \| \pa_t u \|_{L^\infty} 
\lesssim_{s_0} \| (\aff, \tilde u) \|_{X^{s_0+2}},
\end{equation}
and therefore there exists $\rho_1 \in (0,1]$, 
depending only on $s_0$ and on the nonlinearity of the problem, 
such that, for $(\aff, \tilde u)$ in the ball 
\begin{equation} \label{ball.1903.1}
\e^{-1} \| (\aff, \tilde u) \|_{X^{s_0+2}}^p \leq \rho_1,
\end{equation}
the function $u = \e^\s \mS \mT_{\e,c} \aff + \tilde u$ 
satisfies \eqref{ansatz.marzo.1} and \eqref{ball for M}.
Hence Lemma \ref{lemma:inv generale} applies,
and $\tilde h$ satisfies bound \eqref{stima inv generica}. 
Moreover, assuming \eqref{ball.1903.1}, 
the factor in $u$ appearing in \eqref{stima inv generica} 
satisfies
\begin{align}
& (\| u \|_{C^0 L^\infty}^{p-1} \| u \|_{C^1_\e W^{[s]+3}_\e} 
+ \| u \|_{C^0 L^\infty}^\nu \| u \|_{C^0 W^{[s]+1}_\e} \| u \|_{C^1_\e W^2_\e})
\notag \\ & \qquad 
\lesssim_s 
\| (\aff, \tilde u) \|_{X^{s_0+2}}^{p-1} \| (\aff,\tilde u) \|_{X^{[s] + s_0 + 3}}
+ \| (\aff, \tilde u) \|_{X^{s_0+2}}^{\nu+1} \| (\aff,\tilde u) \|_{X^{[s] + s_0 + 1}}
\notag \\ & \qquad 
\lesssim_s
\| (\aff, \tilde u) \|_{X^{s_0+2}}^{p-1} \| (\aff,\tilde u) \|_{X^{s + s_0 + 3}}
\label{u spar}
\end{align}
because $[s] \leq s$, 
$\nu + 1 = \max \{p-2, 0 \} + 1 \geq p-1$ 
and $\| (\aff, \tilde u) \|_{X^{s_0+2}} \leq 1$. 

Thus we have to estimate $f_1$ in \eqref{fg spar}. 
By \eqref{B generica} and \eqref{est R0 Hs}, 
using 
\eqref{u bassa}, 
\eqref{stima u infty}, 
\eqref{def Xs}, 
\eqref{Sobolev y}
and Lemma \ref{lemma:mTeps}, 
for all $s \geq 0$ one has 
\begin{align}
\| \e^{-1} B(u, \e \pa_x) \e^\s \mS \mT_{\e,c} g_2 \|_{H^s_\e} 
& \lesssim_s \e^{\s + \sigmaa - 1} \| (\aff, \tilde u) \|_{X^{s_0+2}}^p \| g_2 \|_{H^{s+1}}
\notag \\ & \quad \ \ 
+ \e^{\s + \sigmaa - 1} \| (\aff, \tilde u) \|_{X^{s_0+2}}^{p-1} 
\| (\aff, \tilde u) \|_{X^{[s]+s_0+1}} \| g_2 \|_{H^1},
\label{stima B g2}
\\
\| R_0(u) \e^\s \mS \mT_{\e,c} g_2 \|_{H^s_\e} 
& \lesssim_s \e^{\s + \sigmaa - 1} \| (\aff, \tilde u) \|_{X^{s_0+2}}^p \| g_2 \|_{H^s}
\notag \\ & \quad \ \ 
+ \e^{\s + \sigmaa - 1} \| (\aff, \tilde u) \|_{X^{s_0+2}}^{p-1} 
\| (\aff, \tilde u) \|_{X^{[s]+s_0+2}} \| g_2 \|_{L^2}.
\label{stima R0 g2}
\end{align}
By \eqref{fg spar}, \eqref{stima B g2}, 
\eqref{stima R0 g2}, \eqref{u spar} and Lemma \ref{lemma:inv generale}, 
for $(\aff, \tilde u)$ in the ball \eqref{ball.1903.1}, for $s \geq 1$ we obtain 
\begin{align}
\| \tilde h \|_{C^1_\e H^s_\e} 
& \lesssim_s \| g_1 \|_{C^0 H^s_\e} 
+ \e^{-1} \| (\aff, \tilde u) \|_{X^{s_0+2}}^{p-1} 
\| (\aff, \tilde u) \|_{X^{s+s_0+3}} \| g_1 \|_{C^0 L^2}
\notag \\ & \quad \ \ 
+ \e^{\s+\sigmaa - 1} \| (\aff, \tilde u) \|_{X^{s_0+2}}^p \| g_2 \|_{H^{s+1}} 
\notag \\ & \quad \ \ 
+ \e^{\s+\sigmaa - 1} \| (\aff, \tilde u) \|_{X^{s_0+2}}^{p-1} 
\| (\aff, \tilde u) \|_{X^{s+s_0+3}} \| g_2 \|_{H^1}.
\label{est tilde h.2}
\end{align}
Since $\bff = g_2$, we get 
\begin{align}
\| (\bff, \tilde h) \|_{X^s} 
& = \e^\s \| \bff \|_{H^s} + \e^{-d/2} \| \tilde h \|_{C^1_\e H^s_\e}
\notag \\ 
& \lesssim_s 
\e^{-d/2} \| g_1 \|_{C^0 H^s_\e} 
+ \e^{-1-d/2} \| (\aff, \tilde u) \|_{X^{s_0+2}}^{p-1} 
\| (\aff, \tilde u) \|_{X^{s+s_0+3}} \| g_1 \|_{C^0 L^2}
\notag \\ & \quad \ \ 
+ \e^\s (1 + \e^{\sigmaa-1-d/2} \| (\aff, \tilde u) \|_{X^{s_0+2}}^p) \| g_2 \|_{H^{s+1}} 
\notag \\ & \quad \ \ 
+ \e^{\s+\sigmaa-1-d/2} \| (\aff, \tilde u) \|_{X^{s_0+2}}^{p-1} 
\| (\aff, \tilde u) \|_{X^{s+s_0+3}} \| g_2 \|_{H^1}
\label{est Psi.3}
\end{align}
for all $(\aff, \tilde u)$ in the ball \eqref{ball.1903.1}. 
As explained in Remark \ref{rem:best rescaling} in general, 
and in Remark \eqref{rem:specific} for our specific problem, 
for $p > 1$ it is convenient 
\begin{itemize}
\item[$(i)$] 
to consider $(\aff, \tilde u)$ in the ball 
\begin{equation} \label{ball.smaller}
\| (\aff, \tilde u) \|_{X^{s_0+2}} \leq \rho_2 \e^{(1+d/2-\sigmaa)/p}, 
\quad \ \text{i.e.} \ \ 
\e^{\sigmaa-1-d/2} \| (\aff, \tilde u) \|_{X^{s_0+2}}^p \leq \rho_1, 
\quad 
\rho_2 := \rho_1^{1/p},
\end{equation}
which is smaller than the ball \eqref{ball.1903.1} if $\sigmaa = 0$,
and it is the same ball if $\sigmaa = d/2$; 

\item[$(ii)$] 
to rescale $\| \ \|_{X^s}$ 
so that \eqref{ball.smaller} becomes a ball with radius $O(1)$ 
(i.e., independent of $\e$) in the rescaled norm.  
\end{itemize}

Thus we define
\begin{equation} \label{def mZs}
\| (\aff, \tilde u) \|_{\mZ^s} 
:= \e^{(\sigmaa - 1 - d/2)/p} \| (\aff, \tilde u) \|_{X^s},
\end{equation}
and \eqref{est Psi.3} becomes
\begin{align}
\| (\bff, \tilde h) \|_{\mZ^s} 
& \lesssim_s 
\e^{-d/2 + (\sigmaa - 1 - d/2)/p} \| g_1 \|_{C^0 H^s_\e} 
+ \e^{\s + (\sigmaa - 1- d/2)/p} (1 + \| (\aff, \tilde u) \|_{\mZ^{s_0+2}}^p) \| g_2 \|_{H^{s+1}} 
\notag \\ & \quad \ \ 
+ \e^{\s + (\sigmaa - 1- d/2)/p} \| (\aff, \tilde u) \|_{\mZ^{s_0+2}}^{p-1} 
\| (\aff,\tilde u) \|_{\mZ^{s + s_0 + 3}}
(\e^{-\s-\sigmaa} \| g_1 \|_{C^0 L^2} + \| g_2 \|_{H^1})
\label{est Psi.4}
\end{align}
for all $s \geq 1$, all $(\aff, \tilde u)$ in the ball 
\begin{equation} \label{ball.2003.1}
\| (\aff, \tilde u) \|_{\mZ^{s_0+2}} \leq \rho_2.
\end{equation}
Therefore, in the case $p>1$, 
\begin{align}
\| (\bff, \tilde h) \|_{\mZ^s} 
& \lesssim_s 
\e^{\s + (\sigmaa - 1- d/2)/p} \big\{ ( \e^{-\s-d/2} \| g_1 \|_{C^0 H^s_\e} + \| g_2 \|_{H^{s+1}})
\notag \\ & \quad \ \ 
+ \| (\aff,\tilde u) \|_{\mZ^{s + s_0 + 3}}
(\e^{-\s-\sigmaa} \| g_1 \|_{C^0 L^2} + \| g_2 \|_{H^1}) \big\}
\label{est Psi.7}
\end{align}
for all $s \geq 1$, all $(\aff, \tilde u)$ in the ball \eqref{ball.2003.1}.

For $p=1$, the restriction to the ball \eqref{ball.smaller} is not convenient
(see Remark \ref{rem:best rescaling} 
and Remark \ref{rem:specific}),  
and we take, instead, $u$ in the entire ball \eqref{ball.1903.1}. 
Hence, for $p=1$, we define 
\begin{equation} \label{def Zs}
\| (\aff, \tilde u) \|_{Z^s} 
:= \e^{-1} \| (\aff, \tilde u) \|_{X^s},
\end{equation}
and \eqref{est Psi.3} becomes 
\begin{align}
\| (\bff, \tilde h) \|_{Z^s} 
& \lesssim_s 
\e^{-1 -d/2} \| g_1 \|_{C^0 H^s_\e} 
+ \e^{-1-d/2} \| (\aff, \tilde u) \|_{Z^{s+s_0+3}} \| g_1 \|_{C^0 L^2}
\notag \\ & \quad \ \ 
+ \e^{\s-1} (1 + \e^{\sigmaa-d/2} \| (\aff, \tilde u) \|_{Z^{s_0+2}}) \| g_2 \|_{H^{s+1}} 
\notag \\ & \quad \ \ 
+ \e^{\s+\sigmaa-1-d/2} \| (\aff, \tilde u) \|_{Z^{s+s_0+3}} \| g_2 \|_{H^1}
\label{est Psi.8}
\end{align}
for all $(\aff, \tilde u)$ in the ball 
\begin{equation} \label{ball.1903.2}
\| (\aff, \tilde u) \|_{Z^{s_0+2}} \leq \rho_2. 
\end{equation}
Therefore, in the case $p=1$, 
\begin{align}
\| (\bff, \tilde h) \|_{Z^s} 
& \lesssim_s 
\e^{\s+\sigmaa-1-d/2} \big\{ (\e^{-\s-\sigmaa} \| g_1 \|_{C^0 H^s_\e} + \| g_2 \|_{H^{s+1}} )
\notag \\ & \quad \ \ 
+ \| (\aff, \tilde u) \|_{Z^{s+s_0+3}} 
(\e^{-\s-\sigmaa} \| g_1 \|_{C^0 L^2} + \| g_2 \|_{H^1}) \big\}
\label{est Psi.9}
\end{align}
for all $s \geq 1$, all $(\aff, \tilde u)$ in the ball \eqref{ball.1903.2}.

Note that we have used norms $\| \ \|_{Z^s}$ for $p=1$ 
and norms $\| \ \|_{\mZ^s}$ for $p > 1$.

\bigskip

\textbf{Estimate for the second derivative.} 
By \eqref{def Xs}, 
\eqref{stima u infty}, 
\eqref{Sobolev y},
and Lemma \ref{lemma:mTeps}, 
any function $u = \e^\s \mS \mT_{\e,c} \aff + \tilde u$ satisfies
\begin{align*}
\| u \|_{H^s_\e} 
& \lesssim_s \e^{\s + \sigmaa} \| \aff \|_{H^s} + \| \tilde u \|_{H^s_\e}
\lesssim_s \e^{\sigmaa} \| (\aff, \tilde u) \|_{X^s},
\\
\| u \|_{L^\infty} 
& \lesssim \| (\aff, \tilde u) \|_{X^{s_0}},
\\ 
\| u \|_{W^{1,\infty}_\e} 
& \lesssim \| (\aff, \tilde u) \|_{X^{s_0+1}}.
\end{align*}
From \eqref{1103.1} we deduce that
\begin{align}
& \| P''(u)[h_1, h_2] \|_{H^s_\e} 
\notag \\ & \quad \quad
\lesssim_s \e^{\sigmaa - 1} \| (\aff, \tilde u) \|_{X^{s_0}}^{p-1} 
\big( \| (\bff_1, \tilde h_1) \|_{X^{s+1}} \| (\bff_2, \tilde h_2) \|_{X^{s_0}}
+ \| (\bff_1, \tilde h_1) \|_{X^{s_0}} \| (\bff_2, \tilde h_2) \|_{X^{s+1}} \big) 
\notag \\ & \quad \quad \quad \ 
+ \e^{\sigmaa - 1} \| (\aff, \tilde u) \|_{X^{s_0}}^{\nu} \| (\aff, \tilde u) \|_{X^{s+1}}
\| (\bff_1, \tilde h_1) \|_{X^{s_0}} \| (\bff_2, \tilde h_2) \|_{X^{s_0}}
\label{est.der.sec.X}
\end{align}
for $u = \e^\s \mS \mT_{\e,c} \aff + \tilde u$,  
$h_i = \e^\s \mS \mT_{\e,c} \bff_i + \tilde h_i$, $i=1,2$, 
and $s \geq 0$. 

With the norms $\| \ \|_{\mZ^s}$ defined in \eqref{def mZs}, 
which we use in the case $p>1$, 
from \eqref{est.der.sec.X} we get
\begin{align*}
& \| P''(u)[h_1, h_2] \|_{H^s_\e} 
\notag \\ & \quad \quad
\lesssim_s \e^{d/2 + (1+d/2-\sigmaa)/p} \| (\aff, \tilde u) \|_{\mZ^{s_0}}^{p-1} 
\big( \| (\bff_1, \tilde h_1) \|_{\mZ^{s+1}} \| (\bff_2, \tilde h_2) \|_{\mZ^{s_0}}
+ \| (\bff_1, \tilde h_1) \|_{\mZ^{s_0}} \| (\bff_2, \tilde h_2) \|_{\mZ^{s+1}} \big) 
\notag \\ & \quad \quad \quad \ 
+ \e^{d/2 + (1+d/2-\sigmaa)/p} \| (\aff, \tilde u) \|_{\mZ^{s_0}}^{\nu} 
\| (\aff, \tilde u) \|_{\mZ^{s+1}}
\| (\bff_1, \tilde h_1) \|_{\mZ^{s_0}} \| (\bff_2, \tilde h_2) \|_{\mZ^{s_0}}.
\end{align*} 
Hence, for $(\aff,\tilde u)$ in the ball \eqref{ball.2003.1}, 
for $s \geq 0$, in the case $p>1$,
one has 
\begin{align}
& \| P''(u)[h_1, h_2] \|_{H^s_\e} 
\notag \\ & \quad \quad 
\lesssim_s \e^{d/2 + (1+d/2-\sigmaa)/p}
\big\{ \| (\bff_1, \tilde h_1) \|_{\mZ^{s+1}} \| (\bff_2, \tilde h_2) \|_{\mZ^{s_0}}
+ \| (\bff_1, \tilde h_1) \|_{\mZ^{s_0}} \| (\bff_2, \tilde h_2) \|_{\mZ^{s+1}} 
\notag \\ & \qquad \quad \ 
+ \| (\aff, \tilde u) \|_{\mZ^{s+1}}
\| (\bff_1, \tilde h_1) \|_{\mZ^{s_0}} \| (\bff_2, \tilde h_2) \|_{\mZ^{s_0}} \big\}.
\label{est.der.sec.mZ}
\end{align} 

For $p=1$,
with the norms $\| \ \|_{Z^s}$ defined in \eqref{def Zs}, 
for $(\aff,\tilde u)$ in the ball \eqref{ball.1903.2}, 
for $s \geq 0$, one has 
\begin{align}
\| P''(u)[h_1, h_2] \|_{H^s_\e} 
& \lesssim_s \e^{\sigmaa + 1} 
\big\{ \| (\bff_1, \tilde h_1) \|_{Z^{s+1}} \| (\bff_2, \tilde h_2) \|_{Z^{s_0}}
+ \| (\bff_1, \tilde h_1) \|_{Z^{s_0}} \| (\bff_2, \tilde h_2) \|_{Z^{s+1}} 
\notag \\ & \quad \ 
+ \| (\aff, \tilde u) \|_{Z^{s+1}}
\| (\bff_1, \tilde h_1) \|_{Z^{s_0}} \| (\bff_2, \tilde h_2) \|_{Z^{s_0}} \big\}.
\label{est.der.sec.Z}
\end{align}

\begin{remark}[\emph{Best rescaling for Nash-Moser application}]
\label{rem:best rescaling}
In this remark we discuss a general, simple way to choose the best rescaling 
to obtain the largest size ball for the solution 
when applying the Nash-Moser Theorem \ref{thm:NMH}
(or essentially any other Nash-Moser theorem).

Suppose we have a nonlinear operator $\Phi$
and a right inverse $\Psi(u)$ of its linearized operator $\Phi'(u)$, 
satisfying an estimate of the form
\begin{equation} \label{rema.1}
\| \Psi(u) g \|_{X^s} \leq (A + B \| u \|_{X^{s_0}}^p) \| g \|_{Y^s} 
+ C \| u \|_{X^{s_0}}^{p-1} \| u \|_{X^s} \| g \|_{Y^{s_0}}
\end{equation}
for all $u$ in a low norm ball
\begin{equation} \label{rema.2}
\| u \|_{X^{s_0}} \leq R
\end{equation}
for some positive constants $A,B,C,R$, 
where $\| \ \|_{X^s}$ are the norms on the domain of $\Phi$, 
$\| \ \|_{Y^s}$ are those on its codomain, 
and $s$ denotes high norms, while $s_0$ denotes low norms
(we ignore any possible loss of regularity, which is not the point 
in this discussion). 
From \eqref{rema.1}, \eqref{rema.2} we deduce bound
\begin{equation} \label{rema.3}
\| \Psi(u) g \|_{X^s} \leq (A + B R^p) \| g \|_{Y^s} 
+ C R^{p-1} \| u \|_{X^s} \| g \|_{Y^{s_0}}
\end{equation}
for $u$ in the ball \eqref{rema.2}. 
Then Theorem \ref{thm:NMH} gives a solution of the problem 
$\Phi(u) = \Phi(0) + g$ for all data $g$ in the ball 
\begin{equation} \label{rema.4}
\| g \|_{Y^{s_0}} \leq \d
\end{equation}
where (ignoring, at least for the moment, the contribution to $\d$ 
coming from the second derivative $\Phi''(u)[h_1, h_2]$ of the operator $\Phi$) 
the radius $\d$ is essentially given by 
\begin{equation} \label{rema.5}
\d = \min \Big\{ \frac{1}{L} , \frac{R}{L} \Big\}, 
\quad \ 
L = A + B R^p + C R^{p-1}.
\end{equation}
Our goal is to find the best (i.e.\ the largest possible) radius $\d$ 
that we can obtain in this situation. 

First, we consider a rescaling of the norm $\| \ \|_{X^s}$: 
for any $\lm$ positive, let 
\begin{equation} \label{rema.6}
\lm \| u \|_{X^s} =: \| u \|_{Z^s}.
\end{equation}
Then \eqref{rema.1}, \eqref{rema.2} become 
\begin{equation} \label{rema.7}
\| \Psi(u) g \|_{Z^s} 
\leq (A \lm + B \lm^{1-p} \| u \|_{Z^{s_0}}^p) \| g \|_{Y^s} 
+ C \lm^{1-p} \| u \|_{Z^{s_0}}^{p-1} \| u \|_{Z^s} \| g \|_{Y^{s_0}}
\end{equation}
for all $u$ in the rescaled ball
\begin{equation} \label{rema.8}
\| u \|_{Z^{s_0}} \leq R \lm.
\end{equation}
From \eqref{rema.7}, \eqref{rema.8} we get the bound
\begin{equation} \label{rema.9}
\| \Psi(u) g \|_{Z^s} 
\leq (A \lm + B \lm R^p) \| g \|_{Y^s} 
+ C R^{p-1} \| u \|_{Z^s} \| g \|_{Y^{s_0}}
\end{equation}
for $u$ in the ball \eqref{rema.8}.
Then Theorem \ref{thm:NMH} solves the nonlinear problem 
for all data $g$ in the ball 
\begin{equation} \label{rema.10}
\| g \|_{Y^{s_0}} \leq \d(\lm), 
\end{equation}
where now the radius is 
\begin{equation} \label{rema.11}
\d(\lm) = \min \Big\{ \frac{R \lm}{L(\lm)}, \frac{1}{L(\lm)} \Big\}, 
\quad \ L(\lm) = \lm (A + B R^p) + C R^{p-1}.
\end{equation}
For $\lm \geq 1/R$, one has
\begin{equation} \label{rema.12}
\d(\lm) = \frac{1}{L(\lm)} 
= \frac{1}{\lm (A + B R^p) + C R^{p-1}}, 
\end{equation}
which is a decreasing function of $\lm$, 
so that $\d(\lm) \leq \d(1/R)$ for all $\lm \geq 1/R$. 
For $0 < \lm \leq 1/R$, one has
\begin{equation} \label{rema.13}
\d(\lm) 
= \frac{R \lm}{L(\lm)}
= \frac{R \lm}{\lm(A + B R^p) + C R^{p-1}} 
= \frac{R}{A + B R^p + C R^{p-1} \lm^{-1}},
\end{equation}
which is an increasing function of $\lm$, 
so that $\d(\lm) \leq \d(1/R)$ for all $\lm \in (0,1/R]$. 
In other words, the largest radius $\d(\lm)$ 
we can get by the rescaling \eqref{rema.6}
is attained at $\lm = 1/R$. 
Note that $\lm = 1/R$ is the value of $\lm$ corresponding to the unit ball 
$\| u \|_{Z^{s_0}} \leq 1$ in the rescaled norm \eqref{rema.8}.
For $\lm = 1/R$ we get the radius 
\begin{equation} \label{rema.14}
\d_R := \d(1/R) 
= \frac{1}{A R^{-1} + (B+C) R^{p-1}}. 
\end{equation}

Second, we check if taking $u$ in a smaller ball can give a better balance 
among the constants, and therefore a larger radius for the data. 
From \eqref{rema.1}, \eqref{rema.2} we deduce that,
for every $r \in (0, R]$,
\begin{equation} \label{rema.15}
\| \Psi(u) g \|_{X^s} \leq (A + B r^p) \| g \|_{Y^s} 
+ C r^{p-1} \| u \|_{X^s} \| g \|_{Y^{s_0}}
\end{equation}
for all $u$ in the ball
\begin{equation} \label{rema.16}
\| u \|_{X^{s_0}} \leq r.
\end{equation}
Apply the best rescaling of the form \eqref{rema.6}, 
which is 
\begin{equation} \label{rema.17}
\frac{1}{r} \| u \|_{X^{s_0}} =: \| u \|_{Z^s}.
\end{equation}
Then, by the discussion above, we obtain the radius
\begin{equation} \label{rema.18}
\d_r = \d(1/r)
= \frac{1}{A r^{-1} + (B+C) r^{p-1}}.
\end{equation}
To maximize the radius $\d_r$ in \eqref{rema.18}, 
we minimize its denominator 
$\ph(r) := A r^{-1} + (B+C) r^{p-1}$
over $r \in (0,R]$. 
For $p=1$, $\ph$ is decreasing in $(0,\infty)$, 
and then the largest $\d_r$ is attained at the largest $r$, 
namely $r=R$. 
For $p>1$, $\ph$ is decreasing in $(0, r_0)$ and increasing in $(r_0,\infty)$, 
where 
\begin{equation} \label{rema.20}
r_0 := \Big( \frac{A}{(p-1)(B+C)} \Big)^{\frac{1}{p}}.
\end{equation}
Hence $\min \{ \ph(r) : r \in (0,R] \}$
is attained at $r=r_0$ if $r_0 \leq R$, 
and at $r=R$ if $R \leq r_0$, 
namely at $r = \min \{ r_0, R \}$ in both cases.
Therefore the best radius is 
\begin{equation} \label{rema.21}
\max_{r \in (0,R]} \d_r
= \begin{cases} 
\d_R & \text{for $p=1$}, \\
\d_R & \text{for $p>1$ and $R \leq r_0$}, \\
\d_{r_0} & \text{for $p>1$ and $r_0 \leq R$}.
\end{cases}
\end{equation}
In fact, to apply the result of this discussion to a specific operator, 
the only point one has to check is whether $r_0 \leq R$ or vice versa.

In this way we get the best radius ignoring the contribution 
coming from $\Phi''(u)$, which is a condition of the form 
$\d \leq M^{-1} L^{-2}$ (see Theorem \ref{thm:NMH}).
Then one has to check if introducing this additional constrain 
to the radius $\d$ does not change its optimal size. 
The heuristic discussion of Section \ref{sec:antani} 
shows that, in many situations, this is the case.
\end{remark}

\begin{remark}
\label{rem:specific}
We see how the discussion of Remark \ref{rem:best rescaling}
applies to our specific problem. 

By \eqref{ball.1903.1} (ignoring the harmless constant $\rho_1$) 
and \eqref{est Psi.3} (ignoring $g_1$, which will be zero in the datum of the original nonlinear problem) one has
\[
A \sim \e^{\s}, \quad 
B \sim C \sim \e^{\s + \sigmaa -1 - d/2}, \quad 
R \sim \e^{1/p}. 
\]
This gives $r_0 \sim \e^{(1+d/2-\sigmaa)/p} \lesssim R$, 
and therefore the best choice is to restrict $u$ 
to the smaller ball $\| u \|_{X^{s_0+2}} \lesssim r_0$ 
and then to rescale as in \eqref{def mZs},
corresponding to $\lm = 1/r_0$.

In the previous case, 
by \eqref{NS.ball.3} and \eqref{NS.energy.h}
one has 
\[
A \sim 1, \quad 
B+C \sim \e^{-pq}, \quad 
R \sim \e^q, 
\]
with $q = 1/p + d/2$.  
This gives $r_0 \sim \e^q \sim R$, 
and therefore the best rescaling for the linearized operator is \eqref{NS.rescaled.1},
corresponding to $\lm = 1/R$. 
\end{remark}

\section{Proof of Theorem \ref{thm:with loss}}
\label{sec:proof with loss}

Let $p>1$, and define
\begin{align} 
E_{\anm,1} & := H^{s_0+\anm}(\R^d), 
\label{2603.X1}
\\
E_{\anm,2} & := \{ \tilde u \in C([0,T], H^{s_0+\anm}(\R^d)) 
\cap C^1([0,T], H^{s_0+\anm-2}(\R^d)) 
: \tilde u(0,x) = 0 \},
\label{2603.X2}
\\
E_\anm & := E_{\anm,1} \times E_{\anm,2}, 
\label{2603.X}
\\
F_{\anm,1} & := C([0,T], H^{s_0+\anm}(\R^d)), 
\label{2603.Y1}
\\ 
F_{\anm,2} & := H^{s_0+\anm+1}(\R^d),
\label{2603.Y2}
\\
F_\anm & := F_{\anm,1} \times F_{\anm,2}. 
\label{2603.Y}
\end{align}
We consider norms \eqref{def mZs} on $E_\anm$, namely 
\begin{equation} \label{recall.def mZs}
\| (\aff, \tilde u) \|_{E_\anm} := \e^{(\sigmaa - 1 - d/2)/p} 
(\e^\s \| \aff \|_{H^{s_0+\anm}} + \e^{-d/2} \| \tilde u \|_{C^1_\e H^{s_0+\anm}_\e}),
\end{equation}
and, on $F_\anm$, we define 
\begin{equation} \label{def Ys}
\| g \|_{F_\anm} 
= \| (g_1, g_2) \|_{F_\anm} 
:= \e^{-\s-d/2} \| g_1 \|_{C^0 H^{s_0+\anm}_\e} + \| g_2 \|_{H^{s_0+\anm+1}}
\end{equation}
(note that $\| \aff \|_{H^{s_0+\anm}}$ and $\| g_2 \|_{H^{s_0+\anm+1}}$ in 
\eqref{recall.def mZs} and \eqref{def Ys} are the standard Sobolev norms, 
without $\e$).
For $(\aff, \tilde u) \in E_\anm$ and $g = (g_1, g_2) \in F_\anm$, we define
\begin{equation} \label{def Sj.coppie}
S_j(\aff, \tilde u) := (S^1_j \aff, \, S^\e_j \tilde u), 
\quad 
S_j g := (S^\e_j g_1, S^1_j g_2),
\end{equation}
where $S^\e_j$, $S^1_j$ are the crude Fourier truncations 
$\e |\xi| \leq 2^j$, $|\xi| \leq 2^j$ respectively, 
namely 
\[
S_j^{\e} f(x) := (2\pi)^{-d/2} \int_{\e |\xi| \leq 2^j} \hat f(\xi) e^{i \xi \cdot x} \, d\xi,
\quad 
S_j^{1} f(x) := (2\pi)^{-d/2} \int_{|\xi| \leq 2^j} \hat f(\xi) e^{i \xi \cdot x} \, d\xi.
\]
Thus $S_j$ in \eqref{def Sj.coppie} satisfy 
all \eqref{S1}-\eqref{2705.4} with constants independent of $\e$. 

We consider the operator $\tilde \Phi$ defined in \eqref{def Phi tilde}. 
The ball \eqref{ball.2003.1} becomes 
\begin{equation} \label{ball.28}
\| (\aff, \tilde u) \|_{E_2} \leq \rho_2.
\end{equation}
For all $(\aff, \tilde u)$ in the ball \eqref{ball.28}, 
by \eqref{est Psi.7} 
the linearized problem $\tilde \Phi'(\aff, \tilde u)(\bff, \tilde h) = g$
has the solution $(\bff, \tilde h) =: \tilde \Psi(\aff, \tilde u) g$, 
which satisfies, for all $\anm \geq 0$,  
\begin{align}
\| \tilde \Psi(\aff, \tilde u) g \|_{E_\anm} 
& \lesssim_s 
\e^{\s + (\sigmaa - 1- d/2)/p} ( \| g \|_{F_\anm}
+ \| (\aff, \tilde u) \|_{E_{\anm + s_0 + 3}} \| g \|_{F_0} ),
\label{est Psi.10}
\end{align}
where we assume that $s_0 \geq 1$ and $s_0 > d/2$.
The second derivatives of $\tilde \Phi$ is
\[
\tilde \Phi''(\aff, \tilde u) [ (\bff_1, \tilde h_1), (\bff_2, \tilde h_2)]
= \begin{pmatrix} P''(u)[h_1, h_2] \\ 0 \end{pmatrix}
\]
where $u = \e^\s \mS \mT_{\e,c} \aff + \tilde u$ 
and $h_i = \e^\s \mS \mT_{\e,c} \bff_i + \tilde h_i$, $i=1,2$. 
By \eqref{est.der.sec.mZ} and \eqref{def Ys}, 
for $(\aff, \tilde u)$ in the ball \eqref{ball.28}, 
one has, for $\anm \geq 0$,  
\begin{align}
& \| \tilde \Phi''(\aff, \tilde u) [ (\bff_1, \tilde h_1), (\bff_2, \tilde h_2)] \, \|_{F_\anm}
= \e^{-\s-d/2} \| P''(u)[h_1, h_2] \|_{C^0 H^{s_0+\anm}_\e} 
\notag \\ & \quad \quad 
\lesssim_s \e^{-\s + (1+d/2-\sigmaa)/p}
\big\{ \| (\bff_1, \tilde h_1) \|_{E_{\anm+1}} \| (\bff_2, \tilde h_2) \|_{E_0}
+ \| (\bff_1, \tilde h_1) \|_{E_0} \| (\bff_2, \tilde h_2) \|_{E_{\anm+1}} 
\notag \\ & \qquad \quad \ 
+ \| (\aff, \tilde u) \|_{E_{\anm+1}}
\| (\bff_1, \tilde h_1) \|_{E_0} \| (\bff_2, \tilde h_2) \|_{E_0} \big\}.
\label{est.der.sec.2603.1}
\end{align} 
Hence $\tilde \Phi$ satisfies the assumptions of Theorem \ref{thm:NMH} 
with
\begin{gather}
\anm_0 = 0, \qquad 
\mu = \anm_1 = 2, \qquad 
\b = \a = s_0+3 > 4, \qquad 
\anm_2 > 2\b - 2, 
\notag \\
U = \{ (\aff,\tilde u) \in E_2 : \| (\aff,\tilde u) \|_{E_2} \leq \rho_2 \}, \qquad 
\d_1 = \rho_2, \qquad 
M_3(\anm) = L_3(\anm) = 0, 
\notag \\
M_1(\anm) = M_2(\anm) = C_\anm \e^{-\s + (1+d/2-\sigmaa)/p}, \quad 
L_1(\anm) = L_2(\anm) = C_\anm \e^{\s - (1+d/2-\sigmaa)/p}. 
\label{LMd.2}
\end{gather}
For any function $\aff_0 = \aff_0(x) \in H^{s_0+\b+1}(\R^d)$, 
the pair $g = (0, \aff_0) \in F_\b$  
trivially satisfies the first inequality in \eqref{2705.1} with $A=1$
(in fact, the inequality is an identity), 
because $\aff_0$ does not depend on the time variable. 
Hence, by Theorem \ref{thm:NMH}, 
for every $g = (0, \aff_0)$ in the ball 
\begin{equation} \label{ball.28.1}
\| \aff_0 \|_{H^{s_0+\b+1}} = \| g \|_{F_\b} \leq \d,  
\end{equation}
with 
\begin{equation} \label{delta.28}
\d = C \e^{-\s + (1 + d/2 - \sigmaa)/p}
\end{equation}
given by \eqref{qui.02}, 
there exists $(\aff, \tilde u) \in E_\a$ such that 
$\tilde \Phi(\aff, \tilde u) = \tilde \Phi(0,0) + g = (0, \aff_0)$. 
By \eqref{def Phi tilde}, this means that $\aff = \aff_0$
and the sum $u = \e^\s \mS \mT_{\e,c} \aff_0 + \tilde u$
solves the nonlinear Cauchy problem \eqref{Cp.general} 
on the time interval $[0,T]$ 
with initial datum $u(0) = u_0 = \e^\s \mT_{\e,c} \aff_0$.
By \eqref{qui.01}, 
\[
\| (\aff, \tilde u) \|_{E_\a} \leq C \e^{\s - (1+d/2-\sigmaa)/p} \| g \|_{F_\b},
\]
namely 
\[
\e^\s \| \aff_0 \|_{H^{s_0+\b}} + \e^{-d/2} \| \tilde u \|_{C^1_\e H^{s_0+\b}_\e} 
\leq C \e^\s \| \aff_0 \|_{H^{s_0+\b+1}},
\]
whence 
\[
\| \tilde u \|_{C^1_\e H^{s_0+\b}_\e} \leq C \e^{\s+d/2} \| \aff_0 \|_{H^{s_0+\b+1}}.
\]
All $\| \aff_0 \|_{H^{s_0+\b+1}} \leq 1$ belong to the ball \eqref{ball.28.1} 
if $1 \leq \d$, 
and this holds for $\e$ sufficiently small if
\[
\s > \frac{1 + d/2 - \sigmaa}{p}.
\]
The higher regularity part of Theorem \ref{thm:with loss} 
is also deduced from Theorem \ref{thm:NMH}.  

Finally, given $s_1 > \max \{ 6 , d+4 \}$, 
we define $s_0 := (s_1 - 4)/2$, 
so that $s_0 > \max \{ 1, d/2 \}$,
and the proof of Theorem \ref{thm:with loss} is complete.
\qed

\bigskip

\begin{remark}[\emph{Confirmation of the heuristics discussion 
of Section \ref{sec:antani} in Theorem \ref{thm:with loss}}] 
\label{rem:confirm.2}
The radius $\d$ given by the Nash-Moser Theorem \ref{thm:NMH} 
is the minimum among $1/L$, $\d_1 / L$, $1 / (L^2 M)$;
here (see \eqref{LMd.2}) 
these three quantities are all of order $\e^{-\s + (1 + d/2 - \sigmaa)/p}$. 
In particular, the ``quadratic condition'' $\d \leq 1 / (L^2 M)$,  
coming from the use of the second derivative $\Phi''(u)$ in the Nash-Moser iteration,
does not modify $\d$.  
This is a confirmation of the heuristic discussion 
of Section \ref{sec:antani}.
\end{remark}

For completeness, now we perform the same analysis in the case $p=1$. 
We consider the same function spaces \eqref{2603.X1}-\eqref{2603.Y} as above, 
but now we use norms \eqref{def Zs} on $E_\anm$, 
namely (see also \eqref{def Xs})
\begin{equation} \label{recall.def mZs.bis}
\| (\aff, \tilde u) \|_{\mE_\anm} := 
\e^{\s-1} \| \aff \|_{H^{s_0+\anm}} 
+ \e^{-1-d/2} \| \tilde u \|_{C^1_\e H^{s_0+\anm}_\e},
\end{equation}
and, on $F_\anm$, we define 
\begin{equation} \label{def Ys.bis}
\| g \|_{\mF_\anm} 
= \| (g_1, g_2) \|_{\mF_\anm} 
:= \e^{-\s-\sigmaa} \| g_1 \|_{C^0 H^{s_0+\anm}_\e} + \| g_2 \|_{H^{s_0+\anm+1}}.
\end{equation}
By \eqref{ball.1903.2}, \eqref{est Psi.9} and \eqref{est.der.sec.Z}, 
for $(\aff, \tilde u)$ in the ball
\begin{equation} \label{ball.28.3}
\| (\aff, \tilde u) \|_{\mE_2} \leq \rho_2,
\end{equation}
for $\anm \geq 0$ one has 
\begin{equation} \label{est.Psi.28}
\| \tilde \Psi(\aff, \tilde u) g \|_{\mE_\anm} 
\lesssim_s \e^{\s + \sigmaa - 1 - d/2} (\| g \|_{\mF_\anm} 
+ \| (\aff, \tilde u) \|_{\mE_{\anm+s_0+3}} \| g \|_{\mF_0})
\end{equation}
and 
\begin{align}
& \| \tilde \Phi''(\aff, \tilde u) [ (\bff_1, \tilde h_1), (\bff_2, \tilde h_2)] \, \|_{\mF_\anm}
= \e^{-\s-\sigmaa} \| P''(u)[h_1, h_2] \|_{C^0 H^{s_0+\anm}_\e} 
\notag \\ & \quad \quad 
\lesssim_s \e^{1-\s}
\big\{ \| (\bff_1, \tilde h_1) \|_{\mE_{\anm+1}} \| (\bff_2, \tilde h_2) \|_{\mE_0}
+ \| (\bff_1, \tilde h_1) \|_{\mE_0} \| (\bff_2, \tilde h_2) \|_{\mE_{\anm+1}} 
\notag \\ & \qquad \quad \ 
+ \| (\aff, \tilde u) \|_{\mE_{\anm+1}}
\| (\bff_1, \tilde h_1) \|_{\mE_0} \| (\bff_2, \tilde h_2) \|_{\mE_0} \big\}.
\label{est.der.sec.28.2}
\end{align} 
Hence $\tilde \Phi$ satisfies the assumptions of Theorem \ref{thm:NMH} 
with
\begin{gather*}
\anm_0 = 0, \qquad 
\mu = \anm_1 = 2, \qquad 
\b = \a = s_0+3 > 4, \qquad 
\anm_2 > 2\b - 2, 
\\
U = \{ (\aff,\tilde u) \in E_2 : \| (\aff,\tilde u) \|_{\mE_2} \leq \rho_2 \}, \qquad 
\d_1 = \rho_2, \qquad 
M_3(\anm) = L_3(\anm) = 0, 
\\
M_1(\anm) = M_2(\anm) = C_\anm \e^{1-\s}, \qquad 
L_1(\anm) = L_2(\anm) = C_\anm \e^{\s+\sigmaa-1-d/2}. 
\end{gather*}
Hence, by Theorem \ref{thm:NMH}, 
for every $g = (0, \aff_0)$ in the ball 
\begin{equation} \label{ball.28.2}
\| \aff_0 \|_{H^{s_0+\b+1}} = \| g \|_{\mF_\b} \leq \d
\end{equation}
with 
\begin{equation} \label{delta.28.2}
\d = C \e^{-\s + 1 + d - 2 \sigmaa}
\end{equation}
given by \eqref{qui.02}, 
there exists $(\aff, \tilde u) \in E_\a$ such that 
$\tilde \Phi(\aff, \tilde u) = (0, \aff_0)$. 
By \eqref{qui.01}, the solution $(\aff, \tilde u)$ satisfies 
\[
\| (\aff,\tilde u) \|_{\mE_\a} \leq C \e^{\s+\sigmaa-1-d/2} \| g \|_{\mF_\b},
\]
namely 
\[
\e^\s \| \aff_0 \|_{H^{s_0+\b}} + \e^{-d/2} \| \tilde u \|_{C^1_\e H^{s_0+\b}_\e} 
\leq C \e^{\s+\sigmaa-d/2} \| \aff_0 \|_{H^{s_0+\b+1}},
\]
whence 
\[
\| \tilde u \|_{C^1_\e H^{s_0+\b}_\e} \leq C \e^{\s+\sigmaa} \| \aff_0 \|_{H^{s_0+\b+1}}.
\]
All $\| \aff_0 \|_{H^{s_0+\b+1}} \leq 1$ belong to the ball \eqref{ball.28.2} 
if $1 \leq \d$, and this holds for $\e$ sufficiently small if
\[
\s > 1 + d - 2 \sigmaa. 
\]

\section{Appendix A. Nash-Moser-H\"ormander implicit function theorem}
\label{sec:NMH}

In this section we state the Nash-Moser-H\"ormander theorem of \cite{BH-NMH}. 

Let $(E_a)_{a \geq 0}$ be a decreasing family of Banach spaces with continuous injections  
$E_b \hookrightarrow E_a$, 
\begin{equation} \label{S0}
\| u \|_{E_a} \leq \| u \|_{E_b} \quad \text{for} \  a \leq b.	
\end{equation}
Set $E_\infty = \cap_{a\geq 0} E_a$ with the weakest topology making the 
injections $E_\infty \hookrightarrow E_a$ continuous. 
Assume that there exist linear smoothing operators $S_j : E_0 \to E_\infty$ for $j = 0,1,\ldots$, 
satisfying the following inequalities, 
with constants $C$ bounded when $a$ and $b$ are bounded, 
and independent of $j$,
\begin{alignat}{2}
\label{S1} 
\| S_j u \|_{E_a} 
& \leq C \| u \|_{E_a} 
&& \text{for all} \ a;
\\
\label{S2} 
\| S_j u \|_{E_b} 
& \leq C 2^{j(b-a)} \| S_j u \|_{E_a} 
&& \text{if} \ a<b; 
\\
\label{S3} 
\| u - S_j u \|_{E_b} 
& \leq C 2^{-j(a-b)} \| u - S_j u \|_{E_a} 
&& \text{if} \ a>b; 
\\ 
\label{S4} 
\| (S_{j+1} - S_j) u \|_{E_b} 
& \leq C 2^{j(b-a)} \| (S_{j+1} - S_j) u \|_{E_a} 
\quad && \text{for all $a,b$.}
\end{alignat}
Set 
\begin{equation}  \label{new.24}
R_0 u := S_1 u, \qquad 
R_j u := (S_{j+1} - S_j) u, \quad j \geq 1.
\end{equation}
Thus 
\begin{equation} \label{2705.3}
\| R_j u \|_{E_b} \leq C 2^{j(b-a)} \| R_j u \|_{E_a} \quad \text{for all} \ a,b.
\end{equation}
Bound \eqref{2705.3} for $j \geq 1$ is \eqref{S4}, 
while, for $j=0$, it follows from \eqref{S0} and \eqref{S2}.
We also assume that 
\begin{equation} \label{2705.4}
\| u \|_{E_a}^2 \leq C \sum_{j=0}^\infty \| R_j u \|_{E_a}^2	\quad \forall a \geq 0,
\end{equation}
with $C$ bounded for $a$ bounded 
(``orthogonality property'' for the smoothing operators).

Suppose that we have another family $F_a$ of decreasing Banach spaces with smoothing operators having the same properties as above. We use the same notation also for the smoothing operators. 

\begin{theorem}[\cite{BH-NMH}] \label{thm:NMH}
\emph{(Existence)} Let $a_1, a_2, \a, \b, a_0, \mu$ be real numbers with 
\begin{equation} \label{ineq 2016}
0 \leq a_0 \leq \mu \leq a_1, \qquad 
a_1 + \frac{\b}{2} \, < \a < a_1 + \b , \qquad 
2\a < a_1 + a_2. 
\end{equation}
Let $U$ be a convex neighborhood of $0$ in $E_\mu$. 
Let $\Phi$ be a map from $U$ to $F_0$ such that $\Phi : U \cap E_{a+\mu} \to F_a$ 
is of class $C^2$ for all $a \in [0, a_2 - \mu]$, with 
\begin{align} 
\|\Phi''(u)[v,w] \|_{F_a} 
& \leq M_1(a) \big( \| v \|_{E_{a+\mu}} \| w \|_{E_{a_0}} 
+ \| v \|_{E_{a_0}} \| w \|_{E_{a+\mu}} \big) 
\notag \\ & \quad 
+ \{ M_2(a) \| u \|_{E_{a+\mu}} + M_3(a) \} \| v \|_{E_{a_0}} \| w \|_{E_{a_0}}
\label{Phi sec}
\end{align}
for all $u \in U \cap E_{a+\mu}$, $v,w \in E_{a+\mu}$,
where $M_i : [0, a_2 - \mu] \to \R$, $i = 1,2,3$, are positive, increasing functions. 
Assume that $\Phi'(v)$, for $v \in E_\infty \cap U$ 
belonging to some ball $\| v \|_{E_{a_1}} \leq \d_1$,
has a right inverse $\Psi(v)$ mapping $F_\infty$ to $E_{a_2}$, and that
\begin{equation}  \label{tame in NM}
\| \Psi(v)g \|_{E_a} \leq 
L_1(a) \|g\|_{F_{a + \b - \a}} + 
\{ L_2(a) \| v \|_{E_{a + \b}} + L_3(a) \} \| g \|_{F_0}
\quad \forall a \in [a_1, a_2],
\end{equation}
where $L_i : [a_1, a_2] \to \R$, $i = 1,2,3$, 
are positive, increasing functions.

Then for all $A > 0$ there exists $\d > 0$ such that, 
for every $g \in F_\b$ satisfying
\begin{equation} \label{2705.1}
\sum_{j=0}^\infty \| R_j g \|_{F_\b}^2 \leq A^2 \| g \|_{F_\b}^2, \quad
\| g \|_{F_\b} \leq \d,
\end{equation}
there exists $u \in E_\a$ solving $\Phi(u) = \Phi(0) + g$.
The solution $u$ satisfies 
\begin{equation} \label{qui.01}
\| u \|_{E_\a} \leq C L_{123}(a_2) (1 + A) \| g \|_{F_\b}, 
\end{equation}
where $L_{123} = L_1 + L_2 + L_3$ 
and $C$ is a constant depending on $a_1, a_2, \a, \b$. 
The constant $\d$ is 
\begin{equation} \label{qui.02}
\d = 1/B, \quad 
B = C' L_{123}(a_2) \max \big\{ 1/\d_1, 1+A, (1+A) L_{123}(a_2) M_{123}(a_2-\mu) \big\}
\end{equation}
where $M_{123} = M_1 + M_2 + M_3$ 
and $C'$ is a constant depending on $a_1, a_2, \a, \b$.  
\\[2mm]
\emph{(Higher regularity)} Moreover, let $c > 0$
and assume that \eqref{Phi sec} holds for all $a \in [0, a_2 + c - \mu]$,
$\Psi(v)$ maps $F_\infty$ to $E_{a_2 + c}$, 
and \eqref{tame in NM} holds for all $a \in [a_1, a_2 + c]$. 
If $g$ satisfies \eqref{2705.1} and, in addition, $g \in F_{\b+c}$ with
\begin{equation} \label{0406.1}
\sum_{j=0}^\infty \| R_j g \|_{F_{\b+c}}^2 \leq A_c^2 \| g \|_{F_{\b+c}}^2 
\end{equation}
for some $A_c$, then the solution $u$ belongs to $E_{\a + c}$, 
with 
\begin{equation} \label{0211.10}	
\| u \|_{E_{\a+c}} \leq C_c \big\{ \mG_1 (1+A) \| g \|_{F_\b} 
+ \mG_2 (1+A_c) \| g \|_{F_{\b+c}} \big\} 
\end{equation}
where 
\begin{align} 
\mG_1 & := \tilde L_3 + \tilde L_{12} (\tilde L_3 \tilde M_{12} + L_{123}(a_2) \tilde M_3) 
(1 + z^{N}), 
\quad \mG_2 := \tilde L_{12} (1 + z^N), 
\label{qui.04}
\\ z & := L_{123}(a_1) M_{123}(0) + \tilde L_{12} \tilde M_{12},
\label{qui.03}
\end{align}
$\tilde L_{12} := \tilde L_1 + \tilde L_2$, 
$\tilde L_i := L_i(a_2+c)$, $i = 1,2,3$; 
$\tilde M_{12} := \tilde M_1 + \tilde M_2$, 
$\tilde M_i := M_i(a_2 + c - \mu)$, $i = 1,2,3$;
$N$ is a positive integer depending on $c,a_1,\a,\b$; 
and $C_c$ depends on $a_1, a_2, \a, \b, c$.
\end{theorem}

\section{Appendix B. Commutator and product estimates}
\label{sec:app}

In the next lemmas we give ``asymmetric'' inequalities 
for the Sobolev norm of commutators and products of functions on $\R^d$, 
with $W^{m,\infty}$ norms ($m$ integer) on one function
and $H^s$ norms ($s$ real) on the other function. 
Estimate \eqref{KP style} is related to the Kato-Ponce inequality 
(see, e.g., \cite{Li}, \cite{Bourgain-Li}, \cite{D'Ancona}), 
but it is not clear how to deduce \eqref{KP style} 
directly from Kato-Ponce.
Hence we give here a proof of \eqref{KP style},
entirely based on well-known estimates.

\begin{lemma} 
\label{lemma:commutator}
Let $s \geq 0$ be real, 
and let $m$ be the smallest positive integer such that $m \geq s$.
Then there exists $C_s$ such that
\begin{equation}\label{KP style}
\| \Lm^s(u v) - u \Lm^s v \|_{L^2} 
\leq C_s ( \| u \|_{W^{1,\infty}} \| v \|_{H^{s-1}} 
+ \| u \|_{W^{m,\infty}} \| v \|_{L^2})
\end{equation}
for all $u \in W^{m,\infty}(\R^d)$, all $v \in H^{s-1}(\R^d) \cap L^2(\R^d)$. 
The constant $C_s$ is increasing in $s$, 
and it is bounded for $s$ bounded.

The same estimate holds with $\Lm^s$ replaced by $\Lm^{s-1} \pa_x^\a$, 
$|\a| = 1$, namely 
\begin{equation}\label{KP var}
\| \Lm^{s-1} \pa_x^\a (u v) - u \Lm^{s-1} \pa_x^\a v \|_{L^2} 
\leq C_s ( \| u \|_{W^{1,\infty}} \| v \|_{H^{s-1}} 
+ \| u \|_{W^{m,\infty}} \| v \|_{L^2}).
\end{equation}
\end{lemma}

\begin{proof}
We use the standard paraproduct decomposition 
$u v = T_u v + (u - T_u) v$ 
(following M\'etivier \cite{Metivier}), 
and split 
\[
\Lm^s(uv) - u \Lm^s v 
= [\Lm^s, T_u] v + \Lm^s ((u-T_u) v) - (u-T_u) \Lm^s v.
\]
The commutator $[\Lm^s, T_u]$ satisfies 
\begin{equation} \label{0503.1}
\| [T_u, \Lm^s] v \|_{L^2} \leq C_s \| u \|_{W^{1,\infty}} \| v \|_{H^{s-1}}
\end{equation}
by Theorem 6.1.4 of \cite{Metivier}. 
The second term satisfies 
\begin{equation} \label{0503.2}
\| \Lm^s ((u-T_u) v) \|_{L^2} 
= \| (u-T_u) v \|_{H^s} 
\leq \| (u-T_u) v \|_{H^m}
\leq C_m \| u \|_{W^{m,\infty}} \| v \|_{L^2}
\end{equation}
by Theorem 5.2.8 of \cite{Metivier}.  
By duality, the third term is also bounded by the r.h.s.\ of \eqref{0503.2}: 
for all $h \in L^2$, by Cauchy-Schwarz,
\[
\langle (u-T_u) \Lm^s v , h \rangle_{L^2}
= \langle v , \Lm^s (u-T_u)^* h \rangle_{L^2}
\leq \| v \|_{L^2} \| (u-T_u)^* h \|_{H^s}
\leq \| v \|_{L^2} \| (u-T_u)^* h \|_{H^m}
\]
where $(u-T_u)^*$ is the adjoint of $(u-T_u)$ with respect to the $L^2$ scalar product. 
Split 
\begin{equation} \label{0603.split}
(u-T_u)^* = (u^* - T_{u^*}) + (T_{u^*} - (T_u)^*).
\end{equation}
The first component in the r.h.s.\ of \eqref{0603.split} satisfies 
\[
\| (u^* - T_{u^*}) h \|_{H^m} 
\leq C_m \| u^* \|_{W^{m, \infty}} \| h \|_{L^2} 
= C_m \| u \|_{W^{m, \infty}} \| h \|_{L^2} 
\]
by Theorem 5.2.8 of \cite{Metivier}. 
The second component in the r.h.s.\ of \eqref{0603.split} satisfies 
\[
\| (T_{u^*} - (T_u)^*) h \|_{H^m} 
\leq C_m \| u \|_{W^{m, \infty}} \| h \|_{L^2} 
\]
by Theorem 6.2.4 of \cite{Metivier}. 
Hence $\| (u-T_u)^* h \|_{H^m}$ is bounded by 
$C_m \| u \|_{W^{m, \infty}} \| h \|_{L^2}$,
and 
\[
\langle (u-T_u) \Lm^s v , h \rangle_{L^2}
\leq C_m \| u \|_{W^{m, \infty}} \| v \|_{L^2} \| h \|_{L^2}  
\]
for all $h \in L^2$. This implies that 
\begin{equation} \label{0503.3}
\| (u-T_u) \Lm^s v \|_{L^2} 
\leq C_m \| u \|_{W^{m, \infty}} \| v \|_{L^2}.
\end{equation}
The sum of \eqref{0503.1}, \eqref{0503.2} and \eqref{0503.3} 
gives \eqref{KP style}.

Similarly, one proves that \eqref{0503.1}, \eqref{0503.2} and \eqref{0503.3} 
also hold with $\Lm^s$ in the l.h.s.\ replaced by $\Lm^{s-1} \pa_x^\a$, $|\a|=1$. 
Then \eqref{KP var} follows.
\end{proof}

\begin{lemma} \label{lemma:prod s real}
Let $s \geq 0$ be real, 
and let $m$ be the smallest positive integer such that $m \geq s$.
Then 
\begin{equation}\label{bona senza eps}
\| u v \|_{H^s} 
\leq 2 \| u \|_{L^\infty} \| v \|_{H^s} 
+ C_s \| u \|_{W^{m,\infty}} \| v \|_{L^2}
\end{equation}
for all $u \in W^{m,\infty}(\R^d)$, all $v \in H^s(\R^d)$. 
The constant $C_s$ is increasing in $s$, 
and it is bounded for $s$ bounded. 

Moreover, for all $0 < \e \leq 1$, 
\begin{equation}\label{bona}
\| u v \|_{H^s_\e} 
\leq 2 \| u \|_{L^\infty} \| v \|_{H^s_\e} 
+ C_s \| u \|_{W^{m,\infty}_\e} \| v \|_{L^2}
\end{equation}
with the same constant $C_s$ as in \eqref{bona senza eps} 
(in particular, $C_s$ is independent of $\e$).
\end{lemma}

\begin{proof} 
By triangular inequality and \eqref{KP style},
\begin{align}
\| uv \|_{H^s} 
= \| \Lm^s(uv) \|_{L^2}
& \leq \| \Lm^s(uv) - u \Lm^s v \|_{L^2} 
+ \| u \Lm^s v \|_{L^2} 
\notag \\ & 
\leq C_s (\| u \|_{W^{1,\infty}} \| v \|_{H^{s-1}}  
+ \| u \|_{W^{m,\infty}} \| v \|_{L^2}) 
+ \| u \|_{L^\infty} \| v \|_{H^s}.
\label{0503.4}
\end{align}
By standard interpolation, 
with $\lm = 1/m$, 
for all $K \geq 1$ one has
\begin{align*}
\| u \|_{W^{1,\infty}} \| v \|_{H^{s-1}}
& \leq \| u \|_{L^\infty}^{1-\lm} \| u \|_{W^{m,\infty}}^\lm 
\| v \|_{H^s}^{1-\lm} \| v \|_{H^{s-m}}^\lm 
\\
& = \frac{1}{K} \big( \| u \|_{L^\infty} \| v \|_{H^s} \big)^{1-\lm}
\big( \| u \|_{W^{m,\infty}} \| v \|_{H^{s-m}} K^m \big)^\lm
\\
& \leq \frac{1}{K} \big( \| u \|_{L^\infty} \| v \|_{H^s} 
+ \| u \|_{W^{m,\infty}} \| v \|_{H^{s-m}} K^m \big)
\\
& \leq \frac{1}{K} \, \| u \|_{L^\infty} \| v \|_{H^s} 
+ K^{m-1} \| u \|_{W^{m,\infty}} \| v \|_{L^2}
\end{align*}
($\| v \|_{H^{s-m}} \leq \| v \|_{L^2}$ 
because $s-m \leq 0$).
We fix $K$ larger or equal to the constant $C_s$ in \eqref{0503.4}, 
and we obtain \eqref{bona senza eps}.

Inequality \eqref{bona} is a straightforward consequence 
of \eqref{bona senza eps}, \eqref{FS.3}, \eqref{FS.6}
and the trivial rescaling identity for the product $R_\e(uv) = (R_\e u)(R_\e v)$.
\end{proof}

\begin{remark} \label{rem:parte intera + 1}
Let $s,m$ be as in Lemmas \ref{lemma:commutator}, \ref{lemma:prod s real}. 
Then $m \leq [s]+1$, where $[s]$ is the integer part of $s$ 
(it is $m=[s]$ for $s$ positive integer, 
and $m = [s]+1$ otherwise).
As a consequence, \eqref{KP style}, \eqref{bona senza eps} and \eqref{bona}
hold with $[s]+1$ in place of $m$. 
\end{remark}

We prove here some elementary inequalities we have used above.

\begin{lemma} \label{lemma:elementary.1}
For every real $s > 0$ there exists $C_s \geq 1$ such that 
\[
(a+b)^s \leq 2 a^s + C_s b^s 
\quad \ \forall a, b \geq 0. 
\]
The constant $C_s$ is increasing in $s$, 
with $C_s = 1$ for $0 < s \leq 1$, 
and $C_s \to \infty$ as $s \to \infty$.
\end{lemma}

\begin{proof}
For $b=0$ the inequality is trivial.
For $b > 0$, divide by $b^s$ and set $\lm = a/b$.  
The inequality holds with best constant
$C_s = \max \{ (1+\lm)^s - 2 \lm^s : \lm \geq 0 \}$, 
which is $C_s = 1$ for $0 < s \leq 1$, 
and $C_s = 2 \cdot (2^{\frac{1}{s-1}} - 1)^{-(s-1)}$ for $s > 1$.
\end{proof}

\begin{lemma} \label{lemma:elementary.2}
For every $s > 0$ there exists $C_s \geq 1$ (increasing in $s$) 
such that 
\[
(1 + (a+b)^2)^s 
\leq 4 (1+a^2)^s + C_s b^{2s}
\quad \ \forall a, b \geq 0.
\]
\end{lemma}

\begin{proof} 
For all $\lm > 0$ one has $2ab = 2(a \lm^{1/2})(b \lm^{-1/2}) 
\leq a^2 \lm + b^2 / \lm$, 
whence 
\[
1 + a^2 + 2ab + b^2 
\leq 1 + a^2 (1 + \lm) + b^2 (1 + 1 / \lm)
\leq (1+a^2)(1+\lm) + b^2 (1 + 1 / \lm).
\]
By Lemma \ref{lemma:elementary.1}, 
\[
(1 + (a+b)^2)^s \leq 2 (1+\lm)^s (1+a^2)^s + C_s (1 + 1 / \lm)^s b^{2s}.
\]
Then we fix $\lm = 2^{1/s}-1$, so that $(1+\lm)^s = 2$ 
and $(1 + 1/\lm)^s = 2 \cdot (2^{1/s} - 1)^{-s}$.
\end{proof}

In the proof of Lemma \ref{lemma:mTeps} we have used 
Lemma \ref{lemma:elementary.2} in the form 
\begin{equation} \label{elem.1}
(1 + |\eta|^2 + 2 |\eta| |\xi_0| + |\xi_0|^2)^s 
\leq 4 (1 + |\eta|^2)^s + C_s |\xi_0|^{2s}, 
\quad \ \eta, \xi_0 \in \R^d.
\end{equation}
Also, by \eqref{elem.1} one directly proves the inequality
\begin{equation} \label{prod.solita}
\| uv \|_{H^s} \leq C_{s_0} \| u \|_{H^{s_0}} \| v \|_{H^s} 
+ C_s \| u \|_{H^s} \| v \|_{H^{s_0}},
\end{equation}
for $s \geq 0$, $s_0 > d/2$, 
which, by rescaling, implies inequality \eqref{FS.ss0}.

\begin{lemma} \label{lemma:prod pax}
For all $s \geq 0$ real,
all functions $u,v$ on $\R^d$, one has 
\begin{align} 
\| u \pa_x v \|_{H^{s-1}} 
& \lesssim_s \| u \|_{L^\infty} \| v \|_{H^s} 
+ \| u \|_{W^{[s]+1,\infty}} \| v \|_{L^2},
\label{prod pax senza eps}
\\
\| u \e \pa_x v \|_{H^{s-1}_\e} 
& \lesssim_s \| u \|_{L^\infty} \| v \|_{H^s_\e} 
+ \| u \|_{W^{[s]+1,\infty}_\e} \| v \|_{L^2},
\label{prod pax}
\end{align}
where $\pa_x$ denotes any $\pa_x^\a$, $|\a|=1$.
\end{lemma}

\begin{proof} 
Write $u \pa_x v$ as $\pa_x (uv) - (\pa_x u) v$. 
For $s \geq 0$, by \eqref{bona senza eps} and Remark \ref{rem:parte intera + 1}, 
\begin{equation} \label{service.1}
\| \pa_x (uv) \|_{H^{s-1}} 
\leq \| uv \|_{H^s}
\lesssim_s \| u \|_{L^\infty} \| v \|_{H^s} 
+ \| u \|_{W^{[s]+1,\infty}} \| v \|_{L^2}.
\end{equation}
For $s \geq 1$, by \eqref{bona senza eps}
and Remark \ref{rem:parte intera + 1}, 
\begin{align}
\| (\pa_x u) v \|_{H^{s-1}} 
& \lesssim_s \| \pa_x u \|_{L^\infty} \| v \|_{H^{s-1}} 
+ \| \pa_x u \|_{W^{[s-1]+1,\infty}} \| v \|_{L^2}
\notag \\
& \lesssim_s \| u \|_{W^{1,\infty}} \| v \|_{H^{s-1}} 
+ \| u \|_{W^{[s]+1,\infty}} \| v \|_{L^2},
\label{service.2}
\end{align}
while for $0 \leq s \leq 1$
\begin{equation} \label{service.3}
\| (\pa_x u) v \|_{H^{s-1}} 
\leq \| (\pa_x u) v \|_{L^2} 
\leq \| \pa_x u \|_{L^\infty} \| v \|_{L^2} 
\leq \| u \|_{W^{1,\infty}} \| v \|_{L^2}.
\end{equation}
The sum of \eqref{service.1} and \eqref{service.3} 
gives \eqref{prod pax senza eps} for $s \in [0,1]$. 
For $s \geq 1$, the sum of \eqref{service.1} and \eqref{service.2} 
gives \eqref{prod pax senza eps} because, by interpolation, 
\[
\| u \|_{W^{1,\infty}} \| v \|_{H^{s-1}} 
\leq \| u \|_{L^\infty} \| v \|_{H^s} 
+ \| u \|_{W^{[s]+1,\infty}} \| v \|_{H^{s-1-[s]}}
\]
and $\| v \|_{H^{s-1-[s]}} \leq \| v \|_{L^2}$. 
Inequality \eqref{prod pax} can be proved similarly, 
or it can be deduced from \eqref{prod pax senza eps}
by rescaling. 
\end{proof}

\begin{lemma}
For all $s \geq 0$ real, one has
\begin{align}
\| [\Lm^s , u] \pa_x v \|_{L^2} 
& \lesssim_s \| u \|_{W^{1,\infty}} \| v \|_{H^s} 
+ \| u \|_{W^{[s]+2,\infty}} \| v \|_{L^2},
\label{0503.5}
\\
\| [\Lm^s_\e , u] \e \pa_x v \|_{L^2} 
& \lesssim_s \| u \|_{W^{1,\infty}_\e} \| v \|_{H^s_\e} 
+ \| u \|_{W^{[s]+2,\infty}_\e} \| v \|_{L^2},
\label{0503.6}
\end{align}
where $\pa_x$ denotes any $\pa_x^\a$, $|\a|=1$. 
\end{lemma}

\begin{proof}
Write
\[
[\Lm^s , u] \pa_x v 
= [\Lm^s \pa_x , u] v - \Lm^s ((\pa_x u) v).
\]
By \eqref{KP var}, $\| [\Lm^s \pa_x , u] v \|_{L^2}$ 
is bounded by the r.h.s.\ of \eqref{0503.5};
by \eqref{bona senza eps}, $\| (\pa_x u) v \|_{H^s}$ 
is bounded by the r.h.s.\ of \eqref{0503.5}. 
Thus \eqref{0503.5} is proved.
Inequality \eqref{0503.6} 
follows from \eqref{0503.5} by rescaling. 
\end{proof}

\begin{lemma} \label{lemma:NS.commu}
For all $s \geq 0$, $s_0 > d/2$, one has
\begin{align}
\| [\Lm^s , u] v \|_{L^2} 
& \lesssim_s \| u \|_{H^{s_0+1}} \| v \|_{H^{s-1}} 
+ \| u \|_{H^s} \| v \|_{H^{s_0}},
\label{NS.commu}
\\
\| [\Lm^s_\e , u] v \|_{L^2} 
& \lesssim_s \e^{-d/2} (\| u \|_{H^{s_0+1}_\e} \| v \|_{H^{s-1}_\e} 
+ \| u \|_{H^s_\e} \| v \|_{H^{s_0}_\e}).
\label{NS.commu.eps}
\end{align}
The same inequalities also hold for $\Lm^{s-1} \pa_x^\a$, 
$\Lm^{s-1}_\e \e \pa_x^\a$, $|\a|=1$, 
in place of $\Lm^s, \Lm^s_\e$ respectively.
\end{lemma}

\begin{proof} 
In the Fourier transform of $[\Lm^s, u] v$ one has 
$\hat u(\xi) \hat v(\eta) \s(\xi,\eta)$, where 
\[
\s(\xi,\eta) 
= \langle \xi + \eta \rangle^s - \langle \eta \rangle^s
= (1 + |\xi + \eta|^2)^{\frac{s}{2}} - (1+|\eta|^2)^{\frac{s}{2}}.
\]
For $|\xi| \leq \frac12 |\eta|$ one has 
$|\s(\xi,\eta)| \lesssim_s \langle \eta \rangle^{s-1} |\xi|$,
leading to the term $\| u \|_{H^{s_0+1}} \| v \|_{H^{s-1}}$ in \eqref{NS.commu}. 
For $|\eta| < 2 |\xi|$ one has 
$|\s(\xi,\eta)| \lesssim_s \langle \xi \rangle^s$, 
leading to the term $\| u \|_{H^s} \| v \|_{H^{s_0}}$ in \eqref{NS.commu}.  
Inequality \eqref{NS.commu.eps} follows by rescaling.
\end{proof}

\begin{lemma} \label{lemma:prod pax Hs}
For all $s \geq 0$ real,
all functions $u,v$ on $\R^d$, one has 
\begin{align} 
\| u \pa_x v \|_{H^{s-1}} 
& \lesssim_s \| u \|_{H^{s_0}} \| v \|_{H^s} 
+ \| u \|_{H^s} \| v \|_{H^{s_0}}
+ \| u \|_{H^{s_0+1}} (\| v \|_{H^{s-1}} + \| v \|_{L^2}),
\label{prod pax Hs senza eps}
\\
\| u \e \pa_x v \|_{H^{s-1}_\e} 
& \lesssim_s \e^{-d/2} \{ \| u \|_{H^{s_0}_\e} \| v \|_{H^s_\e} 
+ \| u \|_{H^s_\e} \| v \|_{H^{s_0}_\e}
+ \| u \|_{H^{s_0+1}_\e} (\| v \|_{H^{s-1}_\e} + \| v \|_{L^2}) \}
\label{prod pax Hs}
\end{align}
where $\pa_x$ denotes any $\pa_x^\a$, $|\a|=1$.
\end{lemma}

\begin{proof} 
We adapt the proof of Lemma \ref{lemma:prod pax}.
Write $u \pa_x v$ as $\pa_x (uv) - (\pa_x u) v$. 
For $s \geq 0$, by \eqref{prod.solita}, 
\begin{equation} \label{NS.service.1}
\| \pa_x (uv) \|_{H^{s-1}} 
\leq \| uv \|_{H^s}
\lesssim_s \| u \|_{H^{s_0}} \| v \|_{H^s} 
+ \| u \|_{H^s} \| v \|_{H^{s_0}}. 
\end{equation}
For $s \geq 1$, by \eqref{prod.solita},
\begin{align}
\| (\pa_x u) v \|_{H^{s-1}} 
& \lesssim_s \| \pa_x u \|_{H^{s_0}} \| v \|_{H^{s-1}} 
+ \| \pa_x u \|_{H^{s-1}} \| v \|_{H^{s_0}}
\notag \\
& \lesssim_s \| u \|_{H^{s_0+1}} \| v \|_{H^{s-1}} 
+ \| u \|_{H^s} \| v \|_{H^{s_0}},
\label{NS.service.2}
\end{align}
while for $0 \leq s \leq 1$
\begin{equation} \label{NS.service.3}
\| (\pa_x u) v \|_{H^{s-1}} 
\leq \| (\pa_x u) v \|_{L^2} 
\leq \| \pa_x u \|_{L^\infty} \| v \|_{L^2} 
\lesssim \| u \|_{H^{s_0+1}} \| v \|_{L^2}.
\end{equation}
Inequality \eqref{prod pax Hs} is deduced from \eqref{prod pax Hs senza eps}
by rescaling. 
\end{proof}

\begin{lemma} \label{lemma:NS.commu.B}
For all $s \geq 0$ real, one has
\begin{align}
\| [\Lm^s , u] \pa_x v \|_{L^2} 
& \lesssim_s \| u \|_{H^{s_0+1}} \| v \|_{H^s} 
+ \| u \|_{H^{s+1}} \| v \|_{H^{s_0}},
\label{NS.commu.B.senza eps}
\\
\| [\Lm^s_\e , u] \e \pa_x v \|_{L^2} 
& \lesssim_s \e^{-d/2} 
(\| u \|_{H^{s_0+1}} \| v \|_{H^s} 
+ \| u \|_{H^{s+1}} \| v \|_{H^{s_0}} )
\label{NS.commu.B}
\end{align}
where $\pa_x$ denotes any $\pa_x^\a$, $|\a|=1$. 
\end{lemma}

\begin{proof}
Write $[\Lm^s , u] \pa_x v = [\Lm^s \pa_x , u] v - \Lm^s ((\pa_x u) v)$.
By Lemma \ref{lemma:NS.commu}, $\| [\Lm^s \pa_x , u] v \|_{L^2}$ 
is bounded by the r.h.s.\ of \eqref{NS.commu.B.senza eps};
by \eqref{prod.solita}, $\| (\pa_x u) v \|_{H^s}$ 
is also bounded by the r.h.s.\ of \eqref{NS.commu.B.senza eps}.
Thus \eqref{NS.commu.B.senza eps} is proved.
Inequality \eqref{NS.commu.B} follows by rescaling. 
\end{proof}

\bigskip

\begin{flushright}
\textbf{Pietro Baldi}

\smallskip

Dipartimento di Matematica e Applicazioni ``R. Caccioppoli''

Universit\`a di Napoli Federico II  

Via Cintia, 80126 Napoli, Italy

\smallskip

\texttt{pietro.baldi@unina.it} 

\bigskip
\smallskip

\textbf{Emanuele Haus}

\smallskip

Dipartimento di Matematica e Fisica,

Universit\`a di Roma Tre 

Largo San Leonardo Murialdo, 
00146 Roma, Italy

\smallskip

\texttt{ehaus@mat.uniroma3.it}
\end{flushright}

\end{document}